\documentclass[12pt]{amsart}

\usepackage{amsmath,amssymb,amsfonts,amsthm,amscd,indentfirst}

\newcommand{\RomanNum}[1]
    {\MakeUppercase{\romannumeral #1}}

\usepackage{indentfirst}
\usepackage{hyperref}
\usepackage{graphicx}

\usepackage{color}
\usepackage[dvipsnames]{xcolor}

\numberwithin{equation}{section}

\usepackage{setspace}
\usepackage{amsmath, amssymb,amsthm, amsfonts, color}

\newtheorem{thm}{Theorem}[section]
\newtheorem{lma}{Lemma}[section]
\newtheorem{prop}{Proposition}[section]
\newtheorem{cor}{Corollary}[section]

\theoremstyle{definition}

\theoremstyle{remark}
\newtheorem{rem}{Remark}[section]

\renewcommand{\div}{\mbox{div}}

\newcommand{\Ric}{\mbox{Ric}}
\newcommand{\R}{\mathbb R}

\newcommand{\be}{\begin{equation}}
\newcommand{\ee}{\end{equation}}
\newcommand{\bee}{\begin{equation*}}
\newcommand{\eee}{\end{equation*}}
\newcommand{\bal}{\begin{aligned}}
\newcommand{\eal}{\end{aligned}}

\def\Vol{\text{Vol}}

\def\p{\partial}

\def\la{\langle}
\def\ra{\rangle}
\def\lf{\left}
\def\ri{\right}
\def\Pi{\displaystyle{\mathbb{II}}}

\def\m{\mathfrak{m}}

\def\c{\mathfrak{c}}

\begin{document}

\title{Mass, capacitary functions, and the mass-to-capacity ratio}

\author{Pengzi Miao}
\address[Pengzi Miao]{Department of Mathematics, University of Miami, Coral Gables, FL 33146, USA}
\email{pengzim@math.miami.edu}

\thanks{P. Miao's research was partially supported by NSF grant DMS-1906423.}

\begin{abstract}

We study connections among the ADM mass, positive harmonic functions tending to zero at infinity, 
and the capacity of the boundary of asymptotically flat $3$-manifolds with nonnegative scalar curvature. 

First we give new formulae that detect the ADM mass via harmonic functions. 
Then we derive a family of monotone quantities and geometric inequalities 
if the underlying manifold has simple topology.
As an immediate application, we observe several additional proofs of the 
$3$-dimensional Riemannian positive mass theorem.
One proof leads to new, sufficient conditions that imply positivity of the mass 
via $C^0$-geometry of regions separating the boundary and $\infty$. 
A special case of such sufficient conditions shows, if a region enclosing 
the boundary has relative small volume, then the mass is positive.

As further applications, we obtain integral identities for 
the mass-to-capacity ratio.
 We also promote the inequalities to become equality on 
spatial Schwarzschild manifolds outside rotationally symmetric spheres. 
Among other things, we show the mass-to-capacity ratio is always bounded below by 
one minus the square root of the normalized Willmore functional of the boundary.

Prompted by our findings, we carry out a study of manifolds satisfying a constraint on the mass-to-capacity ratio. 
We point out such manifolds satisfy improved inequalities, their mass has an upper bound
depending only on the boundary data, there are no closed minimal surfaces enclosing the boundary, 
and these manifolds include static extensions 
in the context of the Bartnik quasi-local mass.

\end{abstract}

\maketitle

\tableofcontents

\markboth{Pengzi Miao}{Mass and capacitary potentials}

\newpage

\section{Introduction and statement of results}

On an asymptotically flat $3$-manifold $(M,g)$, the ADM mass \cite{ADM61} is a flux integral near $\infty$ given by 
$$
\m  = \lim_{ r \to \infty  } \frac{1}{16 \pi} \int_{ |x | = r } \sum_{j, k} ( g_{jk,j} -  g_{jj,k} ) \nu^k  .
$$
Here $\{ x_i \}_{1\le i \le 3} $ is a coordinate chart defining the asymptotic flatness of $(M, g)$
and $\nu = |x|^{-1} x$ denotes the coordinate unit normal to $\{ | x | = r \}$. 
By a result of Bartnik \cite{Bartnik86}, and of Chru\'{s}ciel \cite{Chrusciel86}, 
$\m$ is independent on the choice of the coordinates $\{ x_i \}$.

On an asymptotically flat $3$-manifold $(M,g)$ with boundary $ \Sigma $, 
the capacity (or the $L^2$-capacity) of $\Sigma $ is given by 
$$
\c_{_\Sigma} = \inf \left\{  \frac{1}{4 \pi} \int_M | \nabla f |^2  \right\} , 
$$
where the infimum is taken over all locally Lipschitz functions $f$ that equal $1$ at $\Sigma$ and tend to $0$ at $\infty$. 
Equivalently,
$ \displaystyle \c_{_\Sigma} = \frac{1}{4 \pi} \int_M | \nabla \phi |^2  = \frac{1}{4\pi} \int_{\Sigma} | \nabla \phi|  $, 
where 
$$
\Delta \phi = 0 , \ \ 
 \phi |_\Sigma = 1 , \ \text{and} \  
 \phi \to 0  \ \ \text{at} \  \ \infty.
$$

Regarding the mass, a fundamental result is the Riemannian positive mass theorem, first proved by Schoen and Yau \cite{SchoenYau79}
and by Witten \cite{Witten81}. The theorem states if $(M, g)$ is a complete, asymptotically flat $3$-manifold with nonnegative scalar curvature, without boundary, then $$ \m \ge 0  , $$ 
and equality  holds if and only if $(M, g)$ is isometric to the Euclidean space $ \R^3$. 

Regarding the mass and the capacity, an important result was due to Bray \cite{Bray02}.
Bray showed if  $(M, g)$ is a complete, asymptotically flat $3$-manifold with nonnegative scalar curvature, with minimal surface boundary $\Sigma = \p M$,  then
$$ \m \ge \c_{_\Sigma}  ,$$ 
and equality   holds if and only if  $(M, g)$ is isometric to a spatial Schwarzschild manifold 
outside the horizon. 

If the mean curvature $H$ of the boundary $\Sigma$ in $(M, g)$ is not assumed to be zero, 
using the weak inverse mean curvature flow developed by Huisken and Ilmanen \cite{HI01}, 
Bray and the author \cite{BrayMiao08} showed 
$$  \m \c^{-1}_{_\Sigma} \ge 1 - \left( \frac{1}{16 \pi} \int_\Sigma H^2 \right)^\frac12  $$
under the assumptions 
$ \int_\Sigma H^2 \le 16 \pi $ 
and $H_2 (M, \Sigma) = 0 $, and equality holds if and only if  $(M, g)$ is isometric to a spatial Schwarzschild manifold 
outside a rotationally symmetric sphere with nonnegative (constant) mean curvature.

Recently, level sets of harmonic functions have been found to be an efficient tool to study scalar curvature 
in $3$-dimension. A piorneering work of Stern \cite{Stern19} revealed 
intriguing analogy between the use of such level sets and the use of stable minimal surfaces 
instituted by Schoen and Yau \cite{SchoenYau79}. 
On asymptotically flat $3$-manifolds, a new proof of the positive mass theorem was given by 
Bray, Kazaras, Khuri and Stern \cite{BKKS19}, which made use of harmonic functions 
asymptotic to a linear coordinate function.   

In terms of monotone quantities along the level sets, Munteanu and Wang in \cite{MunteanuWang21} established sharp 
comparison results on complete, 
nonparabolic $3$-manifolds via the discovery of a monotone quantity along level sets of 
the minimal positive Green's function. 
In \cite{AMO21}, Agostiniani, Mazzieri and Oronzio obtained another proof of the Riemannian positive mass theorem 
through a different monotone quantity along level sets of the Green's function on asymptotically flat $3$-manifolds.

In this paper, we consider harmonic functions $u$ satisfying
$$ u (x) = 1 - \c |x|^{-1} + o ( |x|^{-1} ) , \ \text{as} \ x \to \infty , $$ 
for some constant $\c > 0$, on an asymptotically flat 
$3$-manifold $(M,g)$.
In the case $(M, g)$ has boundary $\Sigma$ and $ u $ is $0$ at $ \Sigma$, 
$\c = \c_{_\Sigma}$  and $u$ is referred as the capacitary function on $(M,g)$. 
We obtain a sequence of new results relating the mass of $(M, g)$, the capacitary function 
$u$, and the capacity $\c_{_\Sigma}$.

We first find formulae that detect the mass of $(M,g)$ via the level sets of  such a $u$, see Theorem \ref{thm-limits}. 
In particular, Theorem \ref{thm-limits} (ii) shows
\be \label{eq-intro-limit-1}
\lim_{t \to 1} \frac{1}{ 1 - t }   \left[   4 \pi -  \frac{1}{ (1 - t)^2 } \int_{\Sigma_{t}} | \nabla u |^2   \right] 
= 4 \pi \, \m \, \c^{-1} .
\ee
Here $ \Sigma_t = u^{-1} (t)$.

Besides \eqref{eq-intro-limit-1},  in Theorem \ref{thm-limits} (i), we find 
\be \label{eq-intro-limit-2}
\lim_{t \to 1} \frac{1}{ 1 - t }   \left[   8 \pi -   \frac{1}{ 1 - t} \int_{\Sigma_{t}} H | \nabla u |    \right] 
= 12 \pi \, \m \, \c^{-1} .
\ee
Here $H$ denotes the mean curvature of a regular level set $\Sigma_t$ with respect to $|\nabla u |^{-1} \nabla u $. 

We note that \eqref{eq-intro-limit-1} and \eqref{eq-intro-limit-2} in particular imply 
the ADM mass $\m$ is a geometric invariant of $(M,g)$, since the capacitary function 
$u$ is independent on the coordinates  at $\infty$.

\eqref{eq-intro-limit-1} and \eqref{eq-intro-limit-2} also suggest, as $t \to 1$, 
$$
8 \pi -  \frac{1}{ 1 - t} \int_{\Sigma_{t}} H | \nabla u |  
=  3 \left[ 4 \pi -  \frac{1}{ (1 - t)^2 } \int_{\Sigma_{t}} | \nabla u |^2 \right]  + o \left( ( 1 - t ) \right) .
$$
While this asymptotic comparison was made only via information near $\infty$, 
we show in Theorem \ref{thm-sec-ineq} that, if $M$ has simple topology 
and $g$ has nonnegative scalar curvature, then, at each regular level set $\Sigma_t$, 
\be \label{eq-intro-ineq}
8 \pi -  \frac{1}{ 1 - t} \int_{\Sigma_{t}} H | \nabla u |  
\le  3 \left[ 4 \pi -  \frac{1}{ (1 - t)^2 } \int_{\Sigma_{t}} | \nabla u |^2 \right]  ,
\ee
and ``$=$" holds if and only if $(M,g)$ outside $\Sigma_t$ is isometric to $ \R^3 $ minus a round ball.

Inequality \eqref{eq-intro-ineq} is derived via a monotone quantity along $\{ \Sigma_t \}$, 
see Lemma \ref{lem-monotone-1}. 
Among other things,  we apply \eqref{eq-intro-ineq} to find that 
the quantities in the mass formulae 
\eqref{eq-intro-limit-1} and \eqref{eq-intro-limit-2} 
are actually monotone non-decreasing, that is
\be
\mathcal{A} (t) : =  \frac{1}{1-t} \left[ 8 \pi -   \frac{1}{ 1 - t  } \int_{\Sigma_{t}}  H | \nabla u | \right] 
\ \ \nearrow  \ \  \text{as} \  \ t \nearrow ,
\ee
and 
\be
\mathcal{B} (t) : = \frac{1}{1-t} \left[ 4 \pi -   \frac{1}{ (1 - t )^2 } \int_{\Sigma_{t}}  | \nabla u |^2 \right] 
\  \  \nearrow  \ \  \text{as} \  \ t \nearrow , 
\ee
see Theorem \ref{thm-sec-ineq-2}. Moreover, in Theorem \ref{thm-sec-ineq-2}, 
we show that,  if $u = 0 $ at $ \Sigma = \p M$, then 
\be \label{eq-intro-bdry-ineq-m-a}
8 \pi -  \int_\Sigma H | \nabla u |  \le  12 \pi \, \m \, \c_{_\Sigma}^{-1}  
\ee
and
\be \label{eq-intro-bdry-ineq-m-b}
4 \pi -  \int_\Sigma | \nabla u |^2  \le  4 \pi \, \m \, \c_{_\Sigma}^{-1}  .
\ee
Furthermore, ``$=$" holds in any of these inequalities if and only if $(M,g)$ is isometric to $ \R^3 $ minus a round ball.
 
As an immediate application of \eqref{eq-intro-limit-1} -- \eqref{eq-intro-bdry-ineq-m-b},  
we observe several new arguments implying the $3$-dimensional positive mass theorem, see Section \ref{sec-pmt}.

Inequalities \eqref{eq-intro-bdry-ineq-m-a} and \eqref{eq-intro-bdry-ineq-m-b} also 
give rise to sufficient conditions that imply the positivity of the mass 
via $C^0$-geometry of regions separating the boundary and $\infty$.
For instance, as a special case of Theorem \ref{thm-sec-pmt-bdry}, we show that if $M$ has simply topology and $g$
has nonnegative scalar curvature, then
\be \label{eq-intro-PM-condition}
H \le  \frac{ 8 \pi L^2 }{\Vol(\Omega) } \ \Longrightarrow \ \m > 0 . 
\ee
Here $\Omega $ is a region whose boundary has two components $S_0$ and $S_1$, where $S_1$ encloses $S_0$ and 
$S_0$ encloses $\Sigma$, $L$ is the distance from $ S_1$ to $ S_0 $, and $ \Vol(\Omega)$ is the volume of $(\Omega, g)$.
Another such sufficient condition in Theorem \ref{thm-sec-pmt-bdry-2} shows
\be \label{eq-intro-PM-condition-du}
\int_{S_0} | \nabla v |^2 \le 4 \pi  \ \Longrightarrow \ \m > 0 . 
\ee
Here $ v$ is the harmonic function on $\Omega$ with $ v = 0 $ at $ S_0$ and $ v = 1 $ at $ S_1$.

In \cite{AMO21}, Agostiniani, Mazzieri and Oronzio showed, along $ \{ \Sigma_t \}$,  
\be \label{eq-intro-mono-AMO}
 {F}(t) : =  \frac{1}{ 1 - t }   \left[  4 \pi -   \frac{1}{ 1 - t} \int_{\Sigma_{t}} H | \nabla u |  \ +   \frac{1}{ (1 - t)^2 } \int_{\Sigma_{t}} | \nabla u |^2    \right]  \  \nearrow  \ \text{as} \ t \nearrow  . 
\ee
We observe that $\mathcal{A}(t)$, $\mathcal{B}(t)$ in our work is related to  
 $F(t)$ related by 
\be \label{eq-intro-A-B-F}
 F(t) = \mathcal{A} (t) - \mathcal{B} (t) . 
\ee
In \eqref{eq-difference-mathcal-B} and \eqref{eq-difference-mathcal-A} of the Appendix,  
we give integral identities for the differences 
$$\mathcal{B} (t_2) - \mathcal{B} (t_1), \ \ 
\mathcal{A} (t_2) - \mathcal{A} (t_1), \ \ \text{for} \   t_1 < t_2 . $$
The monotonicity of $F(t)$ can also be seen from 
\eqref{eq-intro-A-B-F}, \eqref{eq-difference-mathcal-B} and \eqref{eq-difference-mathcal-A}. 
Moreover, as a corollary of \eqref{eq-intro-limit-1},  \eqref{eq-intro-limit-2} and  \eqref{eq-intro-A-B-F},  
one has $ \lim_{t \to 1} F (t) = 8 \pi \, \m \c_{_\Sigma}^{-1} $.
Such a limit was shown in \cite{AMO21} in the case that $(M, g)$ is isometric to a spatial Schwarzschild manifold near infinity.

Applying the limits of $\mathcal{A}(t)$, $\mathcal{B}(t)$ as $ t \to 1$ 
and the formulae of  their differences at $ t_1 < t_2 $, 
we derive integral identities for the mass-to-capacity ratio  $ \displaystyle \m \c_{_\Sigma}^{-1}$ 
in  Theorem \ref{thm-sec-mass-integrals}.
Such integral identities can be compared with the mass identity obtained by 
Bray, Kazaras, Khuri and Stern \cite{BKKS19} via harmonic functions having linear asymptotic. 

Inspired by Bray's work \cite{Bray02}, in Section \ref{sec-promoting-ineq} 
we promote inequalities \eqref{eq-intro-ineq}, \eqref{eq-intro-bdry-ineq-m-a} and \eqref{eq-intro-bdry-ineq-m-b} 
to become equality in spatial Schwarzschild spaces.
Among other things, we show in Corollary \ref{cor-sec-more-improved-B-M} that
\be \label{eq-intro-improved-B-M} 
 \left( \frac{1}{\pi} \int_\Sigma  |  \nabla u |^2 \right)^\frac12  \le  \left( \frac{1}{16\pi}  \int_\Sigma   H^2 \right)^\frac12  + 1 .
\ee 
Moreover,  equality holds if and only if $(M, g)$ is isometric to a spatial Schwarzschild manifold  
outside a rotationally symmetric sphere with nonnegative mean curvature. 
In Theorem \ref{thm-sec-more-ineq-1}, we show, given the same triple $(M, g, u)$, 
\be \label{eq-intro-conformal-ineq-m-b-main}
\begin{split}
\frac12  \m  \c_{_\Sigma}^{-1}   \ge 1  -  \left(  \frac{1}{4 \pi}  \int_{\Sigma} | \nabla u |^2 \, d \sigma \right)^\frac12  , 
\end{split}
\ee
and equality holds if and only if 
$(M, g)$ is isometric to a spatial Schwarzschild manifold outside a rotationally symmetric sphere.
As a result of \eqref{eq-intro-improved-B-M} and \eqref{eq-intro-conformal-ineq-m-b-main}, we obtain 
in Theorem \ref{thm-sec-more-ineq-3}
\be \label{eq-B-M-generalization}
\m \c_{_\Sigma}^{-1}  \ge 1 - \left( \frac{1}{16\pi}  \int_\Sigma   H^2 \right)^\frac12 ,
\ee
regardless of the mean curvature $H$ of $\Sigma$. This improves the earlier mentioned 
result of Bray and the author in \cite{BrayMiao08}.

Prompted by \eqref{eq-B-M-generalization}, in Section \ref{sec-m-2-c-less-1} 
we carry out a study of manifolds with boundary satisfying a mass-capacity relation 
\be \label{eq-intro-m-2-c-relation}
 \m \c_{_\Sigma}^{-1} \le 1 . 
\ee
Under this assumption, in Theorem \ref{thm-sec-m-c-l-1} we promote \eqref{eq-intro-bdry-ineq-m-a} to 
\be \label{eq-intro-m-quadratic}
( 2 - \m \c_{_\Sigma}^{-1} ) ( 1 - \m \c_{_\Sigma}^{-1} )  \le \frac{1}{4\pi}  \int_{\Sigma } H | \nabla u | ,
\ee
which picks up an intriguing quadratic term $ ( \m \c_{_\Sigma}^{-1} )^2 $.
Equality in \eqref{eq-intro-m-quadratic} holds if and only if 
$(M, g)$ is isometric to a spatial Schwarzschild manifold outside a rotationally symmetric sphere with
nonnegative mean curvature.

In Corollary \ref{cor-sec-m-c-l-1}, we give a capacity-comparison result for manifolds 
satisfying \eqref{eq-intro-m-2-c-relation} under a condition 
$$ \m H_{max} \le \frac{ 2}{3 \sqrt{3} } .$$
Here $ H_{max} $ is the maximum of the mean curvature of the boundary and the number $ \displaystyle \frac{2}{3 \sqrt{ 3} } $ 
is the maximum value of $ \m H $ evaluated along rotationally symmetric spheres in a spatial Schwarzschild 
manifold with positive mass.

In Corollary \ref{cor-sec-m-c-l-2}, we show manifolds satisfying \eqref{eq-intro-m-2-c-relation} have
the mass bounded  by
\be
\m \le \frac{ r_{_\Sigma}}{2} \left[ 1 + \left( \frac{1}{16 \pi} \int_\Sigma H^2 \right)^\frac12 \right] ,
\ee
where $r_{_\Sigma} $ is the area-radius of $\Sigma$. 
Moreover, the capacitary functions $u$ on these manifolds satisfy 
\be
 \int_{\Sigma} | \nabla u |^2 \ge \pi .
\ee
Heuristically, this suggests such manifolds may not have long cylindrical regions
shielding the boundary, see Remark \ref{rem-long-cylinder}.

Toward the end of Section \ref{sec-m-2-c-less-1}, we place condition \eqref{eq-intro-m-2-c-relation} 
in the context of the Bartnik quasi-local mass \cite{Bartnik89}.
We point out manifolds satisfying \eqref{eq-intro-m-2-c-relation} 
do not contain closed minimal surfaces enclosing the boundary
and static metric extensions with a positive static potential necessarily satisfy \eqref{eq-intro-m-2-c-relation}, 
see Proposition \ref{prop-static}.

We finish this paper with an appendix, including regularization arguments that can be used to verify various monotonicity  
in Section \ref{sec-ineq}.

\vspace{.3cm}

\noindent {\em Acknowledgements.}  
I am indebted to Sven Hirsch for several stimulating conversations related to Sections \ref{sec-mass-integrals} and \ref{sec-promoting-ineq}.  
I also thank Daniel Stern for helpful communications relating to the capacity and level sets of harmonic functions
around the time of a weekly online seminar organized by Hubert Bray at Duke in 2020. 

\section{Detecting the mass at $\infty$} \label{sec-limits}

Let $(M, g)$ denote an asymptotically flat $3$-manifold (with one end) with boundary. 
By this, we mean  
there is a compact set $K \subset M $ such that $M \setminus K$ is diffeomorphic to $\R^3$ minus a ball and, 
with respect to the standard coordinates on $\R^3$, $g$ satisfies
\be \label{eq-AF}
g_{ij} = \delta_{ij} + O ( |x |^{-\tau} ), \ 
 \p g_{ij} = O ( |x|^{-\tau -1}) ,  \ \ \p \p g_{ij} = O (|x|^{-\tau -2} )
\ee
for some constant $ \tau > \frac12 $.
The scalar curvature $R$ of $g$ is also assumed  to be integrable so that the mass $\m $ of $(M,g)$ exists 
(see  \cite{Bartnik86, Chrusciel86} for instance).

Let $ \Sigma$ denote the boundary of $M$.
Let $u$ be the function on $(M, g)$ given by 
\be
\Delta u = 0 \ \text{on} \ M, \  \ u = 0 \ \text{at} \ \Sigma, \ \ \text{and} \ u \to 1 \  \text{at} \  \infty.
\ee
Given any $ t \in [0,1]$, let 
$ \Sigma_t = \{ x \in M \ | \ u (x) = t \} $ denote the level set of $u$.
Below, we collect some basic facts about $u$ and $ \Sigma_t$. 

By the maximum principle, $ \max_{K} u < 1 $, hence $ | x | $ is defined on $ \Sigma_t$ for $t $ close to $1$;
moreover, $ \min_{\Sigma_t} | x | \to \infty $ as $ t \to 1$.
Now suppose $ \tau \in (\frac12, 1)$. 
As $ x \to \infty$, it is known $u$ has an asymptotic expansion 
(see Lemma A.2 in \cite{MMT18} for instance)
\be \label{eq-u-1}
u = 1 -  \c_{_\Sigma} | x |^{-1} + O_2 ( | x |^{-1 -\tau} ) .
\ee
Here $ \c_{_\Sigma} > 0  $ is a positive constant known as the capacity of $\Sigma$ in $(M, g)$. 
Let $ \nabla $ and $ \nabla^2 u $ denote the gradient and the Hessian on $(M, g)$, respectively. 
By \eqref{eq-u-1},
\be \label{eq-nablau-1}
| \nabla u |^2 = \c_{_\Sigma}^2  |x|^{-4}  + O ( |x|^{- 4 - \tau} ) ,
\ee
\be \label{eq-nablau-2}
\begin{split}
(\nabla^2 u )_{ij} = & \ \c_{_\Sigma} |x|^{-3}  \left( - 3 |x|^{-2} x_i x_j + \delta_{ij} \right)+ O ( |x|^{-3 - \tau} ) .
\end{split}
\ee
Thus, $t$ is a regular value if  $t $ is close to $ 1$ and
the mean curvature $H$ of $\Sigma_t $ satisfies 
\be \label{eq-H-est}
H =  \div \left(  | \nabla u |^{-1}  \nabla u  \right) 
=   2  |x|^{-1}    + O (|x|^{-1 - \tau} ) .
\ee
As a result, for $t$ close to $ 1$, $\Sigma_t$ has positive mean curvature and $ \Sigma_t$ is area outer-minimizing as its exterior in $M$ is foliated by mean-convex surfaces $\{ \Sigma_s \}_{s > t }$.

\begin{lma} \label{lem-area}
Let $ | \Sigma_t |$ be the area of $ \Sigma_t $ in $(M, g)$ if $t$ is a regular value of $u$.
Then, as $ t \to 1$, 
\be \label{eq-area-rough}
| \Sigma_t | = 4 \pi \c^2_{_\Sigma} ( 1 - t )^{-2}  + O ( (1 - t)^{ \tau - 2  } ) .
\ee
\end{lma}

\begin{proof}
By \eqref{eq-u-1}, as $ t \to 1$, 
\be \label{eq-x-and-t}
|x| =   \c_{_\Sigma}  (1 -t)^{-1}  + O ( (1 - t )^{ \tau - 1 }  ) .
\ee
Let $r_- (t) = \min_{\Sigma_t} |x| $ and $r_+(t) = \max_{\Sigma_t} |x|$. 
Since $ \Sigma_t$ and the coordinate sphere $ S_r : =  \{ |x|= r \}$ are both area outer-minimizing in $(M, g)$, for $t $ close to $1$ and
large $r$, respectively, we have
\be \label{eq-area-est-1}
| S_{r_- (t) } | \le |\Sigma_t| \le |  S_{r_+ (t)  } | . 
\ee
For large $r$, \eqref{eq-AF} implies 
\be \label{eq-area-coordinate-sphere}
| S_{r} | = 4 \pi r^2 + O ( r^{2 - \tau} ).
\ee
Thus, \eqref{eq-area-rough} follows from \eqref{eq-x-and-t} --  \eqref{eq-area-coordinate-sphere}.
\end{proof}

\begin{lma} \label{lem-limit-1}
As $ t \to 1$, 
\bee
\frac{1}{(1-t)} \int_{\Sigma_t}  H | \nabla u | =  8  \pi + O ( ( 1 - t)^\tau ) 
\eee
and
\bee
 \frac{1}{(1-t)^2} \int_{\Sigma_t } | \nabla u |^2 = 4 \pi + O ( ( 1 - t)^\tau )  .
\eee
\end{lma}

\begin{proof}
By \eqref{eq-H-est} and \eqref{eq-x-and-t}, 
\be \label{eq-H-and-t}
H = 2 \c_{_\Sigma}^{-1} ( 1 - t) + O ( (1-t)^{1+\tau}) .
\ee
Therefore, using the fact $\int_{\Sigma_t} | \nabla u | = 4 \pi \c_{_\Sigma}$, one has
\bee
\begin{split}
\left( \frac{1}{1-t} \int_{\Sigma_t} H | \nabla u | \right)  - 8\pi = & \  \int_{\Sigma_t}  \left( \frac{H}{1-t}  -  \frac{2}{\c_{_\Sigma} } \right) | \nabla u | \\
= & \ O ( (1-t)^\tau ) .
\end{split}
\eee
Similarly,
by \eqref{eq-nablau-1} and \eqref{eq-x-and-t}, 
\be \label{eq-nabu-and-t}
| \nabla u | = \c_{_\Sigma}^{-1} (1-t)^2 + O ( ( 1 - t )^{2 + \tau} ) .
\ee
Therefore,
\bee
\begin{split}
\left( \frac{1}{(1-t)^2} \int_{\Sigma_t}  | \nabla u |^2 \right)  - 4\pi = & \ 
\left(\int_{\Sigma_t}   \frac{| \nabla u |}{(1-t)^2}  - \frac{1}{\c_{_\Sigma} } \right) | \nabla u | \\
= & \ O ( ( 1 - t )^{\tau} ) .
\end{split}
\eee
\end{proof}

\begin{lma} \label{lem-nablau-tangent}
As $ t \to 1$, the gradient of $ | \nabla u | $ on $ \Sigma_t$  satisfies
\be \label{eq-nablau-3}
 | \nabla_{_{\Sigma_t}} | \nabla u | | = O ( | x |^{- 3 -  \tau} ) .
\ee
\end{lma}

\begin{proof}
Write $ \nabla u = (\nabla u)^j  \p_j $. By \eqref{eq-u-1},  
\bee
(\nabla u)^j  = \c_{_\Sigma} | x |^{-2 }  |x|^{-1} x_j + O( | x |^{-2 -\tau} ) .
\eee
Let $ V = V^i \p_i $ denote any unit vector tangent to $\Sigma_t$. 
Then $ V^i = O (1)$ and the fact $ \la V, \nabla u \ra = 0 $ shows
\be \label{eq-V-1}
 \sum_i V^i ( \nabla u )^i  = O ( | x |^{-2 - \tau} ) , \ \text{and hence} \ \sum_i V^i x_i = O ( | x |^{1- \tau} ). 
\ee
Therefore, by  \eqref{eq-nablau-2} and  \eqref{eq-V-1}, 
\be \label{eq-V-u}
\begin{split}
V ( | \nabla u |^2 ) = & \   2 ( \nabla^2 u )( V, \nabla u ) \\
= & \ 2 \c_{_\Sigma} |x|^{-3}  \left[ - 3 |x|^{-2} x_i  V^i  x_j (\nabla u )^j  + \delta_{ij} V^i (\nabla u )^j  \right]   ( 1 + O ( |x|^{ - \tau} ) ) \\
= & \ O ( |x|^{- 5 - \tau} ) .
\end{split}
\ee
Thus, \eqref{eq-nablau-3} follows from \eqref{eq-V-u} and  \eqref{eq-nablau-1}.
\end{proof}

\begin{lma} \label{lem-ringPi}
As $ t \to 1$, the traceless part of the second fundamental form $\Pi$ of $\Sigma_t$, denoted by $ \mathring{\Pi}$, satisfies 
\be \label{eq-ringPi-1}
| \mathring{\Pi} | = O ( | x |^{ - 1 - \tau} ) ,
\ee
and the Gauss curvature $K$ of $\Sigma_t$ satisfies
$
K = |x|^{-2} + O ( |x|^{-2 - \tau} ).
$
\end{lma}

\begin{proof}
Let $ V =V^i \p_i $ and $ W = W^j \p_j $ be any two unit vectors tangent to $\Sigma_t$ at a given point. 
Then
$  \delta_{ij} V^i W^j = g( V, W ) + O ( |x|^{-\tau} ) $.
As $ \nabla^2 u (V, W) = | \nabla u | \Pi ( V, W )$, one has
\be \label{eq-hessian-u-v-w}
\begin{split}
| \nabla u | \,  \Pi ( V, W) = & \  ( \nabla^2 u)_{ij} V^i W^j \\
= & \ \c_{_\Sigma} |x|^{-3}  \left( - 3 |x|^{-2} x_i V^i x_j W^j + \delta_{ij} V^i W^j   \right) ( 1 + O ( |x|^{ - \tau} ) )  \\
= & \ \c_{_\Sigma} |x|^{-3}   g( V, W)   + O ( |x|^{-3 - \tau} ) ,
\end{split}
\ee
where one used \eqref{eq-nablau-2} and \eqref{eq-V-1}.
Therefore, by \eqref{eq-nablau-1}, 
\be \label{eq-Pi-est}
\begin{split}
\Pi (V, W) = & \   |x|^{-1}   g( V, W )   + O ( | x|^{-1 - \tau} ) .
\end{split} 
\ee
This combined with \eqref{eq-H-est} shows 
\bee
\mathring{\Pi} (V, W)  =  \Pi (V, W) - \frac{1}{2} H g( V, W ) 
=  O ( | x|^{-1 - \tau} ) ,
\eee
which proves \eqref{eq-ringPi-1}.
The conclusion on the Gauss curvature follows from  \eqref{eq-Pi-est}, \eqref{eq-H-est} and
the Gauss equation.
\end{proof}

\begin{lma} \label{lem-derivative-Pi}
If $(M, g)$ satisfies  $ \p \p \p g_{ij} = O (|x|^{-3 - \tau} )$ in \eqref{eq-AF}, then
\be \label{eq-derivative-Pi} 
| D \, \mathring \Pi | = O ( | x|^{-2 - \tau} ).
\ee
Here $D$ denote covariant differentiation on $\Sigma_t$. 
\end{lma}

\begin{proof}
If $ g$ satisfies the higher order derivatives decay assumption, then $u$ satisfies 
$$ u = 1 - \c_{_\Sigma} |x|^{-1} + O_3 (|x|^{-1-\tau} ) $$  
(see the proof of Lemma A.2 in \cite{MMT18} for instance). 
The terms $O(|x|^{-3-\tau})$ in \eqref{eq-nablau-2}
and $O (|x|^{-1-\tau})$ in \eqref{eq-H-est} 
are then replaced by $ O_1 ( |x|^{-3-\tau} )$ and $O_1 (|x|^{-1-\tau})$, respectively. 

To prove \eqref{eq-derivative-Pi}, let $\{ V_\alpha \}_{\alpha = 1, 2} $ be a local orthonormal frame around a given point $p$ on $\Sigma_t$. 
By definition, 
$$ ( D_{V_\mu}  \mathring \Pi )  (V_\alpha , V_\beta ) = \left( D_{V_\mu}  \Pi \right)  (V_\alpha , V_\beta ) 
- \frac12  V_\mu ( H )  \delta_{\alpha \beta} . $$
By \eqref{eq-H-est} and  \eqref{eq-V-1},
\bee
\begin{split}
V_{\mu} ( H) = & \ 2 (-1) |x|^{-2} V_\mu (  |x| )  + O ( |x|^{-2 - \tau} ) =  O ( |x|^{-2 - \tau} ) .
\end{split}
\eee
To estimate $ D \Pi$, one may  assume $\{ V_\alpha \} $ is normal at $p$ so that 
\bee
\begin{split}
 ( D_{V_\mu} {\Pi}) (V_\alpha, V_\beta) = & \ V_\mu (  \Pi ( V_\alpha, V_\beta)  ) \\
= & \ V_\mu ( |\nabla u |^{-1} ) \, ( \nabla^2 u )_{\alpha \beta}
+ | \nabla u |^{-1} \, V_\mu \left( ( \nabla^2 u )_{\alpha \beta} \right) . 
\end{split}
\eee
By \eqref{eq-nablau-3} and \eqref{eq-hessian-u-v-w},
$$  V_\mu ( |\nabla u |^{-1} ) \, ( \nabla^2 u )_{\alpha \beta} =  O ( | x|^{-2 - \tau} ) . $$
By \eqref{eq-hessian-u-v-w} and \eqref{eq-V-1},
\bee
 | \nabla u |^{-1} \,  V_\mu \left( ( \nabla^2 u )_{\alpha \beta} \right) = | \nabla u |^{-1}  \, O ( |x|^{-4-\tau} ) 
 = O ( |x|^{-2 - \tau} ) . 
\eee
Thus, $ ( D_{V_\mu} {\Pi}) (V_\alpha, V_\beta)  = O ( | x|^{-2-\tau} )$. 
This proves \eqref{eq-derivative-Pi}. 
\end{proof}

Let $ \m_{_H} ( \Sigma_t) $ denote the Hawking mass \cite{Hawking68} of $\Sigma_t$ 
if $t$ is a regular value of $u$. That is
\be \label{eq-H-mass}
\m_{_H} ( \Sigma_t) = \frac{r_t}{2} \left( 1 - \frac{1}{16 \pi} \int_{\Sigma_t} H^2 \right).
\ee
Here $r_t = \sqrt{ \frac{ | \Sigma_t |}{4\pi} } $ is the area radius of $\Sigma_t$.
By Lemma \ref{lem-area}, 
\be \label{eq-r-and-t}
 r_t =  \c_{_\Sigma} ( 1 - t )^{-1} + O ( (1 - t)^{ \tau - 1 } ).
\ee

\begin{prop} \label{prop-limit-1}
If  $\lim_{ t \to 1} \m_{_H} ( \Sigma_t) = \m $, where  $ \m $ is the mass of $(M, g)$, then
\be \label{eq-prop-limit-1}
\lim_{t \to 1} \frac{1}{ 1 - t }   \left[   8 \pi -  \frac{1}{ 1 - t} \int_{\Sigma_{t}} H | \nabla u |    \right] 
= 12 \pi \, \m \, \c_{_\Sigma}^{-1} 
\ee
and
\be \label{eq-prop-limit-2}
\lim_{t \to 1} \frac{1}{ 1 - t }   \left[   4 \pi -   \frac{1}{ (1 - t)^2 } \int_{\Sigma_{t}} | \nabla u |^2  \right] 
= 4 \pi \, \m \, \c_{_\Sigma}^{-1} .
\ee
\end{prop}

\begin{proof}
For regular values $t$, define 
\be \label{eq-def-At}
A(t) = 8 \pi -   \frac{1}{ 1 - t} \int_{\Sigma_{t}} H | \nabla u |  .
\ee
Then
\bee
 - A'(t) = \frac{1}{ 1 - t  }  \left[ - A(t) + 8 \pi  +  \left( \int_{\Sigma_{t}} H | \nabla u |  \right)'  \, \right] .
\eee
Direct calculation gives
\be \label{eq-Hgrad-prime}
\begin{split}
\left( \int_{\Sigma_t} H | \nabla u | \right)'  
= & \ \int_{\Sigma_t} H' | \nabla u | + H | \nabla u |' + H | \nabla u | H | \nabla u |^{-1}  \\
= & \ \int_{\Sigma_t} - | \nabla u |^{-2} | \nabla_{_{\Sigma_t}} | \nabla u | |^2 + K - \frac34 H^2 - \frac12 | \mathring{\Pi} |^2 - \frac12 R ,
\end{split}
\ee
where $ | \nabla u |' = - H $ and $ H' = - \Delta_{\Sigma_t} | \nabla u |^{-1} - ( \Ric (\nu, \nu) + | \Pi |^2 ) | \nabla u |^{-1}$.

By \eqref{eq-H-mass} and the Gauss-Bonnet theorem, 
\be
8 \pi  +  \left( \int_{\Sigma_{t}} H | \nabla u |  \right)' = 
\frac{24 \pi \m_{_H} (\Sigma_t) }{r_t} - E (t) ,
\ee
where 
$$ E (t) = \int_{\Sigma_t}  | \nabla u |^{-2} | \nabla_{_\Sigma} | \nabla u | |^2 +  \frac{1}{2} | \mathring{\Pi} |^2 + \frac12 R . $$
Therefore,
\be \label{eq-A-ode}
 - A' (t)  =  - \frac{ A(t) }{ 1 - t } + \frac{ 24 \pi \m_{_H} ( \Sigma_t) }{ (1-t) r_t} -  \frac{1}{1 - t} E(t) .
\ee
By Lemma \ref{lem-limit-1}, $ \lim_{t \to 1} A (t) = 0 $. Hence, 
\be \label{eq-A-formula}
A(t) = \frac{1}{1-t}  \int_t^1  \lf[ \frac{24\pi \m_{_H} (s ) }{ r_s} -  E(s) \ri]  .
\ee

As $ t \to 1$, $ \m_{_H} ( \Sigma_t) = \m + o (1)$  by the assumption. Thus, by \eqref{eq-r-and-t},
\be \label{eq-Hawking-to-area}
\frac{\m_{_H} ( \Sigma_t ) }{ r_t} =   \m \, \c^{-1}_{_\Sigma} (1 - t) +  (1 - t ) o (1)  .
\ee
Consequently, 
\be \label{eq-int-mass-r}
\int_t^1 \frac{\m_{_H} ( \Sigma_s ) }{ r_s} = \frac12 \m \, \c_{_\Sigma}^{-1} ( 1 - t)^2 + o ( (1-t)^2) .
\ee

To estimate $ \int_t^1 E(s)$, we note Lemmas \ref{lem-nablau-tangent} and \ref{lem-ringPi}, 
combined with \eqref{eq-x-and-t}, show
\bee
 | \nabla u |^{-2} | \nabla_{_\Sigma} | \nabla u | |^2 +  \frac{1}{2} | \mathring{\Pi} |^2 
 = O ( | x |^{-2 - 2 \tau} ) = O ( ( 1 - t )^{2 + 2 \tau} ) .
\eee
Thus, by Lemma \ref{lem-area}, 
\be \label{eq-Et}
\int_{\Sigma_t} | \nabla u |^{-2} | \nabla_{_\Sigma} | \nabla u | |^2 +  \frac{1}{2} | \mathring{\Pi} |^2
= O ( ( 1 - t )^{ 2 \tau} ) .
\ee
Therefore, 
\be \label{eq-int-error-term}
  \int_{t}^1 \int_{\Sigma_s}  | \nabla u |^{-2} | \nabla_{_\Sigma} | \nabla u | |^2 +  \frac{1}{2} | \mathring{\Pi} |^2   = O ( ( 1 - t )^{1 + 2 \tau} ).
\ee
To handle the scalar curvature term, we use the assumption $R$ is integrable. 
As $ t \to 1$,  
$$ 
o (1) =  \int_{ u \ge t} | R | = \int_t^1 \int_{\Sigma_s} | R | | \nabla u |^{-1}  .
$$ 
By \eqref{eq-nabu-and-t}, 
$ | \nabla u |^{-1} \ge \frac12 \c_{_\Sigma} ( 1 - t )^{-2} $ for $t$ close to $1$.
Hence,
\bee
\begin{split}
\int_t^1 \int_{\Sigma_s} | R | | \nabla u |^{-1} 
\ge & \ \frac12 \c_{_\Sigma} ( 1 - t )^{-2}  \int_t^1 \int_{\Sigma_s} | R | .
\end{split}
\eee
These imply
\be \label{eq-int-scalar}
 \int_t^1 \int_{\Sigma_s} | R | = o ( ( 1 - t )^2 ) . 
\ee

It follows from \eqref{eq-A-formula}, \eqref{eq-int-mass-r},  \eqref{eq-int-error-term}
and \eqref{eq-int-scalar} that
\be \label{eq-expression-At}
\frac{1}{1-t} A(t) = 12 \pi \, \m \, \c_{_\Sigma}^{-1} + o (1) + O \left( (1-t)^{2 \tau - 1 } \right) . 
\ee
Since $ \tau > \frac12 $, this proves \eqref{eq-prop-limit-1}. 

Similarly, define 
\be \label{eq-def-Bt}
B (t) = 4 \pi -   \frac{1}{ (1 - t )^2 } \int_{\Sigma_{t}}  | \nabla u |^2  .
\ee
At any regular value $t$, 
\be \label{eq-B-prime}
\begin{split}
 - B'(t) = & \ \frac{1}{ 1 - t  }  \left[ 2 ( - B(t) + 4 \pi )  + \frac{1}{1-t}  \left( - \int_{\Sigma_t} H | \nabla u |  \right)  \, \right] \\
= & \ \frac{1}{ 1 - t }  \left[ - 2 B(t) + A (t)  \,  \right] .
\end{split}
\ee
By Lemma \ref{lem-limit-1}, $\lim_{t \to 1} B(t) =0 $. Thus, 
\bee
B (t) = \frac{1}{ (1 - t)^2}  \int_t^1 ( 1 - s ) A(s) . 
\eee
Therefore, as $ t \to 1$, by \eqref{eq-expression-At}, 
\bee
\begin{split}
\frac{1}{1-t} B(t)
= & \ 4 \pi \, \m \, \c_{_\Sigma}^{-1} + o ( 1 ) + O \left( (1 - t)^{2 \tau - 1 } \right) .
\end{split}
\eee
This proves \eqref{eq-prop-limit-2}. 
\end{proof}

\begin{thm} \label{thm-limits}
Let $(M, g)$ be an asymptotically flat $3$-manifold with boundary $\Sigma$,
with $ \p\p\p g_{ij} = O ( |x|^{-3-\tau} )$ at $ \infty$.
Let $ u $ be the harmonic function that tends to $1$ at $\infty$ and vanishes at $ \Sigma$.
Then
\vspace{.2cm}
\begin{enumerate}
\item[(i)]
$ \displaystyle 
 \lim_{t \to 1} \frac{1}{ 1 - t }   \left[   8 \pi - \frac{1}{ 1 - t} \int_{\Sigma_{t}} H | \nabla u | \right] 
= 12 \pi \, \m \, \c_{_\Sigma}^{-1} ;$ 
\vspace{.2cm}
\item[(ii)]
$ \displaystyle \lim_{t \to 1} \frac{1}{ 1 - t }   \left[   4 \pi -  \frac{1}{ (1 - t)^2 } \int_{\Sigma_{t}} | \nabla u |^2 \right] 
= 4 \pi \, \m \, \c_{_\Sigma}^{-1} .
$
\vspace{.2cm}
\end{enumerate}
Here $\m $ is the mass of $(M,g)$ and $ \c_{_\Sigma}$ is the capacity of $\Sigma$ in $(M,g)$.
\end{thm}

\begin{proof}
It suffices to show $ \lim_{t \to 1} \m_{_H} (\Sigma_t) = \m $.
For $t$ close to $ 1$, let $ r_- (t)= \min_{\Sigma_t} | x| $ and $ r_+ (t) = \max_{\Sigma_t} |x| $.  
By \eqref{eq-x-and-t}, $ r_+ (t) \le C r_- (t) $. Here and below, $C > 0$ denotes some constant independent on $t$.
By Lemma \ref{lem-area}, $ | \Sigma_t | \le C r_-^2 $.
By Lemma \ref{lem-ringPi}, $ K \ge C r_-^{-2}$, hence $ \text{diam} (\Sigma_t) \le C r_-$.
By Lemma \ref{lem-ringPi} and Lemma \ref{lem-derivative-Pi}, $ | \mathring{\Pi} | \le C  r_-^{-1-\tau} $
and $ | D \mathring{\Pi} | \le C r_-^{-2-\tau}  $. 
Hence, $\{ \Sigma_t \}$ is a family of nearly round surfaces near $\infty$ in $(M,g)$ according to 
Definition 1.3 in \cite{SWW08}. 
By Theorem 2 in \cite{SWW08}, $ \lim_{t \to 1} \m_{_H} (\Sigma_t) = \m $.

Theorem \ref{thm-limits} now follows from Proposition \ref{prop-limit-1}.
\end{proof}

We can indeed interpret the mass-to-capacity ratio 
as the derivatives at $\infty$ of the two functions
\be \label{eq-def-A-and-B}
A(t) = 8 \pi -   \frac{1}{ 1 - t} \int_{\Sigma_{t}} H | \nabla u |  \ \ 
\text{and} \ \ 
B (t) = 4 \pi -   \frac{1}{ (1 - t )^2 } \int_{\Sigma_{t}}  | \nabla u |^2  .
\ee

\begin{cor}
Let $(M, g)$ be an asymptotically flat $3$-manifold with boundary $\Sigma$,
with $ \p\p\p g_{ij} = O ( |x|^{-3-\tau} )$ at $ \infty$.
Let $ u $ be the harmonic function that tends to $1$ at $\infty$ and vanishes at $ \Sigma$.
Then the functions $ A (t)$ and $ B (t)$ have $C^1$ extensions to $ t = 1 $ with 
$$ A (1) = 0 ,  \   A' (1) =  - 12 \pi \, \m \, \c_{_\Sigma}^{-1} , \ 
B (1) = 0, \  B'(1) = - 4 \pi \, \m \, \c_{_\Sigma}^{-1} . $$
\end{cor}

\begin{proof} 
By Lemma \ref{lem-limit-1}, $ A(t)$ and $B(t)$  extend continuously to $ t = 1$ 
with $ A(1) =0 $ and $ B(1)= 0$. 
By Theorem \ref{thm-limits} (i),   \eqref{eq-A-ode}, \eqref{eq-Hawking-to-area} and \eqref{eq-Et} ,
\bee
\begin{split}
\lim_{t \to 1} A'(t) = & \ \lim_{t \to 1}  \left[\frac{1}{1-t} A(t) - \frac{ 24 \pi \m_{_H} (\Sigma_t ) } { (1-t) r_t} 
+ \frac{1}{1-t} E(t) \right] \\
= & \ 12 \pi \, \m \, \c_{_\Sigma}^{-1} - 24 \pi \, \m \, \c_{_\Sigma}^{-1} \\
= & \  \lim_{t \to 1} \frac{1}{t-1} A(t) .
\end{split}
\eee
Similarly, by Theorem \ref{thm-limits} (i), (ii) and \eqref{eq-B-prime}, 
\bee
\lim_{t \to 1} B'(t) = \lim_{t \to 1} \frac{1}{1-t} \left[ 2 B(t) - A(t)  \right] = - 4\pi \, \m \, \c_{_\Sigma}^{-1} 
= \lim_{ t \to 1} \frac{1}{t-1} B(t) . 
\eee
This shows $ A'(t) $ and $B'(t)$ are continuous at $ t = 1 $ with $A'(1) = - 12 \pi \, \m \, \c_{_\Sigma}^{-1}$
and $ B'(1) = - 4\pi \, \m \, \c_{_\Sigma}^{-1}  $.
\end{proof}

\section{Inequalities along the level sets} \label{sec-ineq}

In this section, we establish a family of geometric inequalities 
along $\{ \Sigma_t \}$ under assumptions that $g$ has nonnegative scalar curvature and $M$ has 
simple topology.

We first compare  
\bee 
A(t) = 8 \pi -  \frac{1}{ 1 - t} \int_{\Sigma_{t}} H | \nabla u |   \ \ 
\text{and} \  \ B(t) = 4 \pi -  \frac{1}{ (1 - t)^2 } \int_{\Sigma_{t}} | \nabla u |^2  .
\eee

\begin{thm} \label{thm-sec-ineq}
Let $(M, g)$ be a complete, orientable, asymptotically flat $3$-manifold with boundary $\Sigma$.
Suppose $ \Sigma$ is connected and $H_2 (M, \Sigma) = 0$.
Let $ u $ be the harmonic function that tends to $1$ at $\infty$ and vanishes at $ \Sigma$.
If $ g$ has nonnegative scalar curvature, then
\be \label{eq-bdry-ineq-0}
4 \pi  +  \frac{1}{1-t} \int_{\Sigma_t} H | \nabla u |  \ge  \frac{3}{ (1-t)^2 } \int_{\Sigma_t} | \nabla u |^2 
\ee
for all regular values $t$, and equality holds at some $t$ if and only if  $(M, g)$, outside $\Sigma_t$, 
is isometric to $ \R^3$ minus a round ball.

In particular, at $ \Sigma$, 
\be \label{eq-bdry-ineq}
4 \pi + \int_\Sigma H | \nabla u |  \ge  3 \int_\Sigma | \nabla u |^2  ,
\ee
and equality holds if and only if  $(M, g)$ is isometric to $ \R^3$ minus a round ball.
\end{thm}

\begin{rem} 
Inequality \eqref{eq-bdry-ineq-0} is equivalent to 
\be \label{eq-A-and-3B}
A(t) \le 3 B(t) .
\ee
We will use Theorem \ref{thm-sec-ineq} in this form later to 
derive other inequalities along $ \{ \Sigma_t \}$. 
\end{rem}

To prove Theorem \ref{thm-sec-ineq}, we begin with a lemma which may be derived directly from the work of Stern in \cite{Stern19}.

\begin{lma} \label{lem-monotone-1}
Let $ (\Omega, g)$ be a compact, orientable, Riemannian $3$-manifold with nonnegative scalar curvature, 
with boundary $\p \Omega$.
Suppose $\p \Omega$ has two connected components $S_1$ and $ S_2 $.
Let $ u $ be a harmonic function on $(\Omega, g)$ such that $ u = c_i$ on $ S_i$, $ i = 1, 2$, 
where $ c_1$, $ c_2$ are constants with $ c_1 < c_2 < 1$.
If  the level set $\Sigma_s : = u^{-1} (s)$ is connected for $s \in [c_1, c_2]$, then
\bee
\Psi (t) : =  4 \pi ( 1 - t ) + \int_{\Sigma_t} H | \nabla u |   - \frac{3}{1-t} \int_{\Sigma_t} | \nabla u |^2 
\  \searrow  \ \text{as} \ t \nearrow ,
\eee
i.e. $\Psi (t)$ is monotone nonincreasing.
Here $ t \in [ c_1, c_2]$ denotes a regular value of $u$ and $H$ is the mean curvature of $ \Sigma_t $ with respect to 
the unit normal $  \nu =  | \nabla u |^{-1} \nabla u $.
\end{lma}

\begin{proof}
Let $ t_1 < t_2 $ be two regular values of $ u$. 
On $ \Omega_{[t_1, t_2] } : = \{ x \in \Omega \ | \ t_1 \le u(x) \le t_2 \}$, one has
\be \label{eq-bdry-1}
\begin{split}
\int_{\Sigma_{t_1} } H | \nabla u |  -  \int_{\Sigma_{t_2} } H | \nabla u |   \ge  \ 
\int_{t_1}^{t_2}   \int_{\Sigma_t} \frac{1}{2}  \left(   \frac{ | \nabla^2 u |^2 }{  | \nabla u |^2  }  +  R   \right)
  -  2 \pi  \int_{t_1}^{t_2} \chi(\Sigma_t)  .
\end{split}
\ee
Here $ \nabla^2 u$, $\nabla u $ denote the Hessian, the gradient of $u$ on $(M, g)$, respectively,  
$R$ is the scalar curvature of $g$, and $\chi (\Sigma_t)$ is the Euler characteristic of $\Sigma_t$. 
Relation \eqref{eq-bdry-1} is a direct consequence of Stern's computations in Section 2 of \cite{Stern19},
and can also be found explicitly from  (4.7) in \cite{BKKS19} and (2.18) in \cite{HMT20}.

Let $ \Pi$ denote the second fundamental form of $ \Sigma_t$ w.r.t $\nu$. 
Along $ \Sigma_t$,  
\be \label{eq-lap-conse}
\nabla^2 u  (X, Y)  = | \nabla u | \Pi (X, Y), \  \nabla^2 u  (X, \nu) = X ( | \nabla u | ), \  
\nabla^2 u (\nu, \nu) = - H | \nabla u | ,
\ee
where $X, Y$ denote vectors tangent to $ \Sigma_t$ and the last equation follows from $ \Delta u = 0 $.
Thus,
\be \label{eq-hessian-square}
| \nabla u |^{-2} \, | \nabla^2 u |^2 =   | \Pi |^2 + 2 | \nabla u |^{-2} | \nabla_{_{\Sigma_t}} | \nabla u | |^2 
+ H^2    .
\ee
Here $\nabla_{_{\Sigma_t} } $ denotes the gradient on $\Sigma_t$.
Under the assumption $\Sigma_t$ is connected, 
it follows from \eqref{eq-bdry-1} and \eqref{eq-hessian-square} that
\be \label{eq-bdry-3}
\begin{split}
& \  4 \pi ( t_2 - t_1) + \int_{\Sigma_{t_1} } H | \nabla u |  -  \int_{\Sigma_{t_2} } H | \nabla u |   \\
 \ge  & \   \int_{t_1}^{t_2}   \int_{\Sigma_t}  \frac{1}{2}  | \mathring{\Pi} |^2 +  | \nabla u |^{-2}  | \nabla_{_{\Sigma_t} } | \nabla u | |^2 + \frac34 H^2  + \frac12 R  ,
\end{split}
\ee
where $ \mathring{\Pi}$ denotes the traceless part of $\Pi$.

To handle the term of $ H^2$ in \eqref{eq-bdry-3}, we follow the idea in \cite{MunteanuWang21, AMO21} 
to replace it with $ \left( H - 2  | \nabla u | ( 1 - u)^{-1} \right)^2$. 
A motivation to this may be seen in the model case in which $\Omega = \{ R_1 \le |x| \le R_2 \} \subset \R^3$ and 
$u = 1 - |x|^{-1}$. In this special setting, 
$H$ and $|\nabla u | $ satisfy  $ H =  2 | \nabla u | ( 1 - u)^{-1}$ along any level set sphere.

Thus, one can rewrite \eqref{eq-bdry-3} as 
\be \label{eq-bdry-4}
\begin{split}
& \ 4 \pi ( t_2 - t_1) + \int_{\Sigma_{t_1} } H | \nabla u |  -  \int_{\Sigma_{t_2} } H | \nabla u |  \\
& \ +  3   \int_{t_1}^{t_2} \left[  - \frac{1}{1 - t}  \int_{\Sigma_t}   H  | \nabla u | 
+  \frac{1}{ ( 1 - t)^2 } \int_{\Sigma_t}  | \nabla u |^2 \right] \\
 \ge  & \   \int_{t_1}^{t_2}  \int_{\Sigma_t} \frac{1}{2}  | \mathring{\Pi} |^2 +  | \nabla u |^{-2}  | \nabla_{_{\Sigma_t} } | \nabla u | |^2  +  \frac{3}{4}  \left( H -  \frac{ 2 | \nabla u | }{ 1 - u } \right)^2  + \frac12 R   .
\end{split}
\ee

At each regular value $t$, one has 
$  \left( \int_{\Sigma_t} | \nabla u |^2 \right)' = - \int_{\Sigma_t} H | \nabla u | $,
and therefore,
\bee
\begin{split}
 - \frac{1}{1 - t}  \int_{\Sigma_t}   H  | \nabla u | 
+  \frac{1}{ ( 1 - t)^2 } \int_{\Sigma_t}  | \nabla u |^2 
 = \frac{d}{dt} 
\left( \frac{1}{1- t}  \int_{\Sigma_t}  | \nabla u |^2 \right) .
\end{split}
\eee
Thus, if $[t_1, t_2]$ has no critical values, the above directly shows
\be \label{eq-integral-relation}
\begin{split}
& \  \int_{t_1}^{t_2} \left(   - \frac{1}{1 - t}  \int_{\Sigma_t}    H  | \nabla u | 
+  \frac{1}{ ( 1 - t)^2 } \int_{\Sigma_t}  | \nabla u |^2  \right) \\
= & \ 
 \frac{1}{1- t_2}  \int_{\Sigma_{t_2}}  | \nabla u |^2 
-  \frac{1}{1- t_1}  \int_{\Sigma_{t_1}}  | \nabla u |^2  .
\end{split}
\ee

In general, if $[t_1, t_2]$ has critical values, one may use a regularization argument to still obtain \eqref{eq-integral-relation}.
For instance, applying Lemma \ref{lem-reg} of the Appendix to $u$ on $ \Omega_{[t_1, t_2] }$, one has
\be \label{eq-integral-relation-2}
\begin{split}
& \ \frac{1}{1 - t_2}  \int_{\Sigma_{t_2} }  | \nabla u |^2   
-  \frac{1}{1 - t_1} \int_{\Sigma_{t_1}}   | \nabla u |^2 \\
= & \ \int_{\Omega_{[t_1, t_2]} }  \frac{ | \nabla u |^3  }{ ( 1 - u )^2 }   + 
\int_{ \{ \nabla u \ne 0  \}   \subset \Omega_{[t_1, t_2]} } \frac{1} { 1 - u} | \nabla u |^{-1}  \nabla^2 u ( \nabla u , \nabla u  ) .
\end{split}
\ee
This, together with the coarea formula and \eqref{eq-lap-conse}, gives \eqref{eq-integral-relation}.

By \eqref{eq-bdry-4}  and \eqref{eq-integral-relation},
\be \label{eq-Psi-explicit}
\Psi (t_1) - \Psi  (t_2) \ge  
 \int_{t_1}^{t_2}  \int_{\Sigma_t} \frac{1}{2}  | \mathring{\Pi} |^2 +  | \nabla u |^{-2}  | \nabla_{_{\Sigma_t} } | \nabla u | |^2  +  \frac{3}{4}  \left( H -  \frac{ 2 | \nabla u | }{ 1 - u } \right)^2  + \frac12 R  .
\ee
For the later purpose in Section \ref{sec-regularization}, we note that 
\eqref{eq-Psi-explicit} holds without assumptions on the scalar curvature $ R$.

If the scalar curvature $R$ is nonnegative, then \eqref{eq-Psi-explicit} implies $ \Psi (t_1) \ge \Psi(t_2)$,
which proves the Lemma. 
\end{proof}

In the context of Theorem \ref{thm-sec-ineq}, the assumption $\Sigma$ is connected and  
$H_2 (M, \Sigma) = 0 $ is a sufficient condition to ensure $\chi (\Sigma_t) \le 2$ for a regular $\Sigma_t$.
Under this condition, $u$ being harmonic and the maximum principle guarantee
$\Sigma_t$ is connected. 
(The same assumption was used by Bray and the author \cite{BrayMiao08}
in estimating the capacity of $\Sigma$ in $(M, g)$ via the solution to the weak inverse mean curvature ($1/H$) flow \cite{HI01}. In that setting, a different reasoning shows the level set of the $1/H$ flow is connected.) 

\begin{proof}[Proof of Theorem \ref{thm-sec-ineq}]
Let $ \Psi (t)$ be given from Lemma \ref{lem-monotone-1}. 
On an asymptotically flat $(M, g)$, a corollary of Lemma \ref{lem-limit-1} shows
\bee
\lim_{t \to 1}  \Psi (t)  = 0  .
\eee
Thus, letting $ t_2 \to 1 $ in \eqref{eq-Psi-explicit} gives 
\be \label{eq-Psi-expression}
\begin{split}
\Psi (t) \ge  & \ 
 \int_{t}^{1}  \int_{\Sigma_s} \frac{1}{2}  | \mathring{\Pi} |^2 +  | \nabla u |^{-2}  | \nabla_{_{\Sigma_t} } | \nabla u | |^2  +  \frac{3}{4}  \left( H -  \frac{ 2 | \nabla u | }{ 1 - u } \right)^2  + \frac12 R   
\end{split}
\ee
for every regular value $t$. 
In particular, if $ R \ge 0 $, then  $\Psi (t) \ge 0$.

Inequality \eqref{eq-bdry-ineq-0} follows from \eqref{eq-Psi-expression} by noting that
\be \label{eq-Psi-B-A}
\frac{1}{1-t} \Psi (t) = 3 B(t) - A (t) .
\ee

To show the rigidity case of \eqref{eq-bdry-ineq-0}, it suffices to establish it for the case $ t = 0 $.
Suppose the equality in \eqref{eq-bdry-ineq} holds, 
then, by \eqref{eq-Psi-expression} and its proof, 
for every regular value $t \in [0,1]$, $\Sigma_t$ is connected (orientable) with $\chi (\Sigma_t) = 2$,
hence $ \Sigma_t$ is a $2$-sphere; moreover, $R =0 $,  $ | \nabla u | $ only depends on $t$,  $\Sigma_t$ is totally umbilic, 
and $ H = \frac{ 2  }{ 1 - t } | \nabla u | $.

To show $(M, g)$ is isometric to $ \R^3$ minus a round ball, we start from a neighborhood of the boundary 
$\Sigma $.
For convenience, we normalize $(M, g)$ so that $| \Sigma | = 4 \pi $. 
It follows from the equality 
$$ 4 \pi + \int_\Sigma H | \nabla u | = 3 \int_\Sigma | \nabla u |^2  $$
that $ | \nabla u | = 1 $ and $H = 2 $ at $ \Sigma = \Sigma_0$.
Locally, $g$ takes the form of 
$ g = \eta (t) ^{-2} d t^2 + \gamma_t $ near $\Sigma_0$, 
where $t = u $, $ \eta (t) = | \nabla u | $ and $\gamma_t$ denotes the induced metric on $\Sigma_t$, 
which satisfies 
$
\p_t \gamma_t = 2 \eta(t)^{-1} \Pi_t =  \eta(t)^{-1} H \gamma_t = 2(1-t)^{-1} \gamma_t . 
$
Thus, $ (1 - t)^2 \gamma_t = $ a fixed metric. Similarly, since $ | \nabla u |' = - H $, 
$\eta (t)$ satisfies  $  \eta'(t) = \frac{- 2}{1-t} \eta (t) $. Hence, $ ( 1 - t )^{-2} \eta = $ a constant. 
As $ | \nabla u | = 1 $ at $ \Sigma$, we thus have $ \eta = (1-t)^2$ and
$ g =  (1-t)^{-4} d t^2 +  (1-t)^{-2} \sigma_o $ 
for some fixed metric $\sigma_o$ on the $2$-sphere $\Sigma$.
Invoking  the fact $ R = 0 $ near $\Sigma$, we see $\sigma_o$ is a round metric with Gauss curvature $1$ on $\Sigma$.

Now, if $u$ has a critical value, let $ t_0 \in (0,1) $ be the smallest critical value of $u$. 
The above argument then shows 
$ ( u^{-1} ([0, t_0) ) , g )$ is isometric to 
$$ \left( \Sigma \times [0, t_0) , (1-t)^{-4} d t^2 +  (1-t)^{-2} \sigma_o \right) . $$
In particular, this implies $ | \nabla u | = (1- t_0)^{2} \ne 0 $ on the set $ \p \{ u < t_0 \} = \p \{ u \ge t_0 \}  $.
As a result, $  \p \{ u \ge t_0 \}$ is an embedded surface in $M$. Therefore, $  \p \{ u \ge t_0 \} = \{ u = t_0 \} $ 
by the strong maximum principle. 
In summary, this shows $ \nabla u \ne 0 $ on the set $\{ u = t_0 \}$, 
which contradicts to the assumption $t_0$ is a critical value. 
Hence, $u$ has no critical values. We conclude $(M, g)$ is isometric to 
$$ \left( \Sigma \times [0, 1) , (1-t)^{-4} d t^2 +  (1-t)^{-2} \sigma_o \right) ,$$
which, upon a change of variable $ 1 - t = r^{-1}$, 
is isometric to $ \R^3 $ minus a unit ball.
\end{proof}

Theorem \ref{thm-sec-ineq} implies an upper bound of 
$ \frac{1}{(1-t)^2} \int_{\Sigma_t}  | \nabla u |^2 $ 
via $ \int_{\Sigma_t} H^2 $.  
  
\begin{cor} \label{cor-sec-ineq}
Let $(M, g)$ be a complete, orientable, asymptotically flat $3$-manifold with boundary $\Sigma$.
Suppose $ \Sigma$ is connected and $H_2 (M, \Sigma) = 0$.
Let $ u $ be the harmonic function such that $ u = 0 $ at $ \Sigma$ and $ u \to 1$ at $\infty$.
If $ g$ has nonnegative scalar curvature, then
\be \label{eq-int-gradu-2-est}
\frac{1}{4\pi} \int_\Sigma | \nabla u |^2  \le \frac{1}{9} \left[ 2 W + 2 \sqrt{ W^2 + 3 W } + 3 \right] ,
\ee
where $ W =   \frac{1}{16 \pi} \int_{\Sigma} H^2  $, 
and equality holds if and only if  $(M, g)$ is isometric to $ \R^3$ minus a round ball.
\end{cor}

\begin{proof}
Let $ z =  \left( \int_\Sigma | \nabla u |^2  \right)^\frac12 $.
By Theorem \ref{thm-sec-ineq} and H\"{o}lder's inequality, 
\bee
4 \pi + \sqrt{ 16 \pi W}  z \ge 3 z^2 .
\eee
This implies the bound of $z$ in \eqref{eq-int-gradu-2-est} by elementary reason. The equality case follows from the equality case in Theorem \ref{thm-sec-ineq}.
\end{proof}

We next apply Theorem \ref{thm-sec-ineq} to show that 
the quantities  in Theorem \ref{thm-limits}, which approach  to constant multiples of $\m \c_{_\Sigma}^{-1}$ at $\infty$,
are actually monotone.

\begin{thm} \label{thm-sec-ineq-2}
Let $(M, g)$ be a complete, orientable, asymptotically flat $3$-manifold with boundary $\Sigma$.
Suppose $ \Sigma$ is connected and $H_2 (M, \Sigma) = 0$.
Let $ u $ be the harmonic function such that $ u = 0 $ at $ \Sigma$ and $ u \to 1$ at $\infty$. 
If $ g$ has nonnegative scalar curvature, then
\begin{enumerate}
\item [(i)] 
$ \displaystyle 
 \mathcal{A}(t): = \frac{1}{1-t} \left[ 8 \pi -   \frac{1}{ 1 - t  } \int_{\Sigma_{t}}  H | \nabla u | \right] 
\  \nearrow  \ \text{as} \ t \nearrow ,
$
i.e. $\mathcal{A}(t)$ is monotone non-decreasing in $t$.
As a result, 
$
\mathcal{A} (t) \le 12 \pi \, \m \, \c_{_\Sigma}^{-1} .
$
In particular, at $\Sigma$, 
\be \label{eq-bdry-ineq-m-a}
8 \pi -  \int_\Sigma H | \nabla u |  \le  12 \pi \, \m \, \c_{_\Sigma}^{-1}  ,
\ee
and equality holds if and only if  $(M, g)$ is isometric to $ \R^3$ minus a round ball.

\item [(ii)] 
$ \displaystyle 
 \mathcal{B}(t): = \frac{1}{1-t} \left[ 4 \pi -   \frac{1}{ (1 - t )^2 } \int_{\Sigma_{t}}  | \nabla u |^2 \right] 
\  \nearrow  \ \text{as} \ t \nearrow ,
$
i.e. $\mathcal{B}(t)$ is monotone non-decreasing in $t$.
As a result, 
$ 
\mathcal{B} (t) \le 4 \pi \, \m \, \c_{_\Sigma}^{-1} .
$ 
In particular, at $\Sigma$, 
\be \label{eq-bdry-ineq-m-b}
4 \pi -  \int_\Sigma | \nabla u |^2  \le  4 \pi \, \m \, \c_{_\Sigma}^{-1}  ,
\ee
and equality holds if and only if  $(M, g)$ is isometric to $ \R^3$ minus a round ball.
\end{enumerate}
\end{thm}

\begin{proof}
We first show (ii) as it is more straightforward. 
By \eqref{eq-B-prime} and \eqref{eq-A-and-3B}, at every regular value $t$, we have
\bee
\begin{split}
 - B'(t) = \frac{1}{ 1 - t }  \left[ - 2 B(t) + A (t)  \,  \right] \le \frac{1}{1-t} B(t) .
\end{split}
\eee
Therefore, 
$ \displaystyle \left[ \frac{1}{1-t} B(t) \right]' \ge 0 , $ 
which implies the monotonicity of $\mathcal{B} (t) = \frac{1}{1-t} B(t)$
in the case that $u$ has no critical values.  
If $u$ has critical values, we may again apply a regularization argument
to show that $\mathcal{B}(t_2) - \mathcal{B}(t_1) \ge 0 $ for $ t_2 > t_1$,
see Proposition \ref{prop-reg-monotone} in the Appendix for details. 

By Theorem \ref{thm-limits} (ii), 
$$ \lim_{t \to 1} \mathcal{B} (t) = 4 \pi \, \m \, \c_{_\Sigma}^{-1} .$$
Therefore, the monotonicity of $\mathcal{B} (t)$ shows 
$$ \mathcal{B} (t) \le 4 \pi \, \m \, \c_{_\Sigma}^{-1} . $$
At $ t =0 $, this gives
$$ \mathcal{B} (0) = 4 \pi - \int_{\Sigma}  | \nabla u |^2  \le 4 \pi  \, \m \, \c_{_\Sigma}^{-1} . $$
The rigidity part follows from the rigidity part of Theorem \ref{thm-sec-ineq}.

To show (i), we calculate
\bee
\begin{split}
\mathcal{A}'(t) = & \ \frac{1}{(1- t)^2} \left[  A (t)  -  \frac{1}{1-t}   \int_{\Sigma_t} H | \nabla u |  
- \left( \int_{\Sigma_t} H | \nabla u | \right)' \right] .
\end{split}
\eee
By  \eqref{eq-Hgrad-prime} and the Gauss-Bonnet theorem, 
\bee
\begin{split}
\mathcal{A}'(t) \ge & \ \frac{1}{(1- t)^2} \left[  A (t) - \frac{1}{1-t} \int_{\Sigma_t} H | \nabla u | 
+  \int_{\Sigma_t}  \left( - K +  \frac34 H^2  \right) \right] \\
\ge & \   \frac{1}{(1- t)^2} \left[  A (t) - \frac{1}{1-t} \int_{\Sigma_t} H | \nabla u | 
- 4 \pi +  \frac34 \int_{\Sigma_t}  H^2 \right] \\
=  & \   \frac{1}{(1- t)^2} \left[ \underbrace{4 \pi + \frac{1}{1-t} \int_{\Sigma_t} H | \nabla u | 
- \frac{3}{(1-t)^2} \int_{\Sigma_t} | \nabla u |^2 }_{I(t)}  \right. \\
& \ \ \ \ \ \ \ \ \ \ \ \ \  \left. +   \frac34 \int_{\Sigma_t}  \left( H - \frac{2 | \nabla u | }{ 1 - u } \right)^2 \right] .
\end{split}
\eee
By \eqref{eq-A-and-3B}, 
$$ I(t) = 3 B(t) - A (t) \ge 0  . $$
Therefore, $ \mathcal{A}'(t) \ge 0 $, which implies the monotonicity of $\mathcal{A}(t)$
in the absence of critical values. 
The general case can be again handled by a regularization argument that shows 
$\mathcal{A}(t_2) - \mathcal{A}(t_1) \ge 0 $ for $ t_2 > t_1$,
see Proposition \ref{prop-reg-monotone} in the Appendix. 

The remaining conclusions in (i) follow from Theorem \ref{thm-limits} (i) and Theorem \ref{thm-sec-ineq}.
\end{proof}

\begin{rem}
We need the technical assumption $\p \p \p g_{ij} = O (|x|^{-\tau -3} ) $ in obtaining 
\eqref{eq-bdry-ineq-m-a} and \eqref{eq-bdry-ineq-m-b} as Theorem \ref{thm-limits} is used in that step. 
For convenience, we include this assumption in the asymptotic flatness description \eqref{eq-AF} henceforth. 
\end{rem}

\begin{rem} \label{rem-comaring}
Comparing \eqref{eq-bdry-ineq}, \eqref{eq-bdry-ineq-m-a} and \eqref{eq-bdry-ineq-m-b}, 
we have $ \eqref{eq-bdry-ineq}  +  \eqref{eq-bdry-ineq-m-b}  \Rightarrow   \eqref{eq-bdry-ineq-m-a} $.
\end{rem}

\section{Proofs of the positive mass theorem} \label{sec-pmt}

The $3$-dimensional Riemannian positive mass theorem, first proved by 
Schoen-Yau \cite{SchoenYau79} and later by Witten \cite{Witten81},
asserts that if $(M, g)$ is a complete, asymptotically flat $3$-manifold with nonnegative scalar curvature, 
then $\m \ge 0 $ and $ \m = 0 $ if and only if $(M^3, g)$ is isometric to $ \R^3$.

Since the work of Schoen-Yau and Witten, 
other proofs of this theorem have been given by Huisken-Ilmanen \cite{HI01}, 
by Li \cite{Li18}, by Bray-Kazaras-Khuri-Stern \cite{BKKS19}, and by Agostiniani-Mazzieri-Oronzio \cite{AMO21}. 
(Agostiniani-Mantegazza-Mazzieri-Oronzio \cite{AMMO22} also gave a new proof of 
the Riemannian Penrose inequality, first proved by Bray \cite{Bray02} and Huisken-Ilmanen \cite{HI01}.)

As applications of Theorems \ref{thm-limits} and \ref{thm-sec-ineq-2}, 
we observe a few additional arguments that prove the positive mass theorem (PMT). 
We first outline the tools and features of the proofs to be given: 
\begin{itemize}

\vspace{.2cm}

\item Proof \RomanNum{1} uses Theorem \ref{thm-limits} (ii) and 
a result of Munteanu-Wang \cite{MunteanuWang21}.

\vspace{.2cm}

\item Proof \RomanNum{2} is self-contained. It makes use of  Theorem \ref{thm-limits} and 
Theorems \ref{thm-sec-ineq-2}. 

\vspace{.2cm}

\item Proof \RomanNum{3} is self-contained. It uses the inequalities in Theorem \ref{thm-sec-ineq-2}. 
Proof \RomanNum{3} leads to new sufficient conditions that guarantee the positivity of the mass, see Section \ref{sec-pmt-bdry}.

\vspace{.2cm}

\end{itemize}

\begin{proof}[Proof \RomanNum{1}]
Let $(M, g)$ be a complete, asymptotically flat $3$-manifold without boundary, with nonnegative 
scalar curvature. Suppose $M$ is topologically $ \R^3$.

Take  $ p \in M$. Let $ G (x) $ be the minimal positive Green's function with a pole at $p$, with $ G (x) \to 0 $ as $ x \to \infty$. 
Let $ u = 1 - G $. 
By Theorem 1.1 of Muntenau-Wang \cite{MunteanuWang21},  
\bee
\begin{split} 
 4\pi ( 1 - t ) -  \frac{1}{1 - t} \int_{\Sigma_t} | \nabla u |^2 
\  \searrow  \ \text{as} \ t \nearrow , 
\end{split}
\eee
i.e. it is monotone non-increasing in $t$. 

As $ t \to 1$,  $ \frac{1}{1 - t} \int_{\Sigma_t} | \nabla u |^2 \to 0 $  by Lemma \ref{lem-limit-1}.
Hence,  
$ 4\pi ( 1 - t ) -  \frac{1}{1 - t} \int_{\Sigma_t} | \nabla u |^2 \ge 0 $.
Consequently, 
\bee
\frac{1}{ (1-t) } \left[ 4\pi - \frac{1}{ (1 - t)^2 } \int_{\Sigma_t} | \nabla u |^2 \right] \ge 0 .
\eee
By Theorem \ref{thm-limits} (ii) and 
the fact $ u = 1 - \frac{1}{ 4 \pi } |x|^{-1} + O (|x|^{-1-\tau}) $, 
$$ \lim_{t \to 1} \frac{1}{ (1-t) } \left[ 4\pi - \frac{1}{ (1 - t)^2 } \int_{\Sigma_t} | \nabla u |^2 \right]  
= ( 4 \pi)^2  \m . $$
Hence $ \m \ge 0 $.
\end{proof}

\begin{rem}
To prove the $3$-dimensional positive mass theorem,  it is known it suffices to assume $M$ is topologically $ \R^3$, 
see Section 2 in \cite{BKKS19} for instance.
For this reason, we make such an assumption in all the proofs.
It is also known the rigidity case $m = 0 $ in the theorem follows from the inequality $\m \ge 0$ 
by a variational argument, see \cite{SchoenYau79}.
\end{rem}

\vspace{0.1cm}

\begin{proof}[Proof \RomanNum{2}]
Take $ p \in M$. Let $ G (x) $ be the minimal positive Green's function with a pole at $p$.
Let $ d (x)$ denote the distance from $x$ to $ p $ in $(M, g)$. 
As $ x \to p$, it is known
\be \label{eq-expansion-G-x}
 G (x) = \frac{1}{4\pi}  d(x)^{-1} + o ( d(x)^{-1} ) ,  \ 
| \nabla G (x) | = \frac{1}{4 \pi} d(x)^{-2} + o ( d (x)^{-2} ) .
\ee

Consider $ u = 1 - G $. By Theorem \ref{thm-sec-ineq-2} (ii),  
\be \label{eq-PM-monotonicity}
 \mathcal{B}(t) = \frac{1}{1-t} \left[ 4 \pi -   \frac{1}{ (1 - t )^2 } \int_{\Sigma_{t}}  | \nabla u |^2 \right] 
\  \nearrow  \ \text{as} \ t \nearrow ,
\ee
i.e. it is monotone non-decreasing in $t$. 
Note this is different from the monotonicity of Munteanu-Wang \cite{MunteanuWang21}.
The latter asserts $ (1-t)^2 \mathcal{B} (t)$ is monotone non-increasing.

As $ t \to - \infty$, by \eqref{eq-expansion-G-x}, 
$ \frac{1}{ (1 - t )^2 } \int_{\Sigma_{t}}  | \nabla u |^2 $ is bounded, hence 
$ \frac{1}{ (1 - t )^3 } \int_{\Sigma_{t}}  | \nabla u |^2 \to 0$.
Thus,
\be \label{eq-PM-limit-p}
 \lim_{t \to - \infty} \mathcal{B} (t) = 0 . 
\ee
Hence, by \eqref{eq-PM-monotonicity} and \eqref{eq-PM-limit-p}, 
$ \mathcal{B} (t) \ge 0 $. By Theorem \ref{thm-limits} (ii), 
$$ (4 \pi)^2 \, \m  = \lim_{t \to 1} \mathcal{B} (t) \ge 0 . $$ 
\end{proof}

\begin{rem}
Proof \RomanNum{2} is similar to the proof of Agostiniani-Mazzieri-Oronzio \cite{AMO21}.
The difference is the use of different monotone quantities, i.e. $\mathcal{B}(t)$ compared to $F(t)$.
A feature of the quantity $\mathcal{B} (t)$  is that it does not involve derivatives of the metric. 
\end{rem}

\begin{rem} \label{rem-work-A}
One can also work with $\mathcal{A}(t)$, and apply Theorem \ref{thm-sec-ineq-2} (i)
and Theorem \ref{thm-limits} (i). In this case, one checks 
$ \lim_{t \to - \infty} \frac{1}{(1-t)^2} \int_{\Sigma_t} H | \nabla u | = 0 $, which
follows from the known estimate on $ \nabla^2 G $ near the pole 
(see \cite{MunteanuWang21} and \cite{AMO21} for instance).
\end{rem}

\vspace{.1cm}

\begin{proof}[Proof \RomanNum{3}]
Take $ p \in M$. Given a small $r > 0 $, let $ B_r $ denote the geodesic ball 
of radius $r $ centered at $p$. Let $ \Sigma_r = \p B_r $ and $u = u_r$ 
be the harmonic function with $ u=0 $ at  $\Sigma_r $ and $ u \to 1$ at $\infty$.
Let $ \c_{r}$ be the capacity of $\Sigma_r$ in $(M, g)$. 

Applying  \eqref{eq-bdry-ineq-m-a} of Theorem \ref{thm-sec-ineq-2} (i) to $(M \setminus B_r , g)$, we have
\bee
8 \pi - \int_{\Sigma_r} H | \nabla u | \le 12 \pi \, \m \, c_{r}^{-1} .
\eee
Since $ c_r > 0$, this is equivalent to
\be \label{eq-pmt-ineq-i}
c_r \left( 8 \pi - \int_{\Sigma_r} H | \nabla u | \right) \le 12 \pi \, \m  .
\ee
It remains to check, as $ r \to 0$,  
\be \label{eq-small-cap-and-Hu}
 c_r = O (r) \ \text{and} \  \int_{\Sigma_r} H | \nabla u | = O (1) .
\ee
A conclusion $\m \ge 0$ will follow from \eqref{eq-pmt-ineq-i} and \eqref{eq-small-cap-and-Hu}.

To estimate $c_r$, we may use the variational characterization of the capacity, 
i.e. 
\be \label{eq-var-capacity}
 c_r = \inf_{f} \left\{ \frac{1}{4\pi} \int_{ M \setminus B_r} | \nabla f |^2 \right\} ,
\ee
where $f$ is a Lipschitz function with $ f= 0$ at $\Sigma_r$ and  $ f \to 1$ at $\infty$.
Consider a test function $ f (x) = r^{-1} \left(  d (x) - r \right) $
in $ B_{2 r} \setminus B_r$ and extend $f$ to be $1$ outside $B_{2 r}$.
Here $d (x)$ is the distance from $x $ to $p$. Then 
\be \label{eq-capacity-small-sphere}
\begin{split}
c_r \le & \ \frac{1}{4\pi} \int_{B_{2r} \setminus B_r} | \nabla f |^2 
=  \frac{1}{4\pi r^2 } \, \text{Volume}(B_{2 r \setminus B_r} )  = O (r) .
\end{split}
\ee

For  $\int_{\Sigma_r} H | \nabla u | $, we have
\bee
\left|  \int_{\Sigma_r} H | \nabla u | \right|  \le \max_{\Sigma_r} |H | \, \int_{\Sigma_r} | \nabla u | = 
\max_{\Sigma_r} |H | \, \c_r  = O (1)
 \eee
by \eqref{eq-capacity-small-sphere} and the fact $ H = 2 r^{-1} + O (r) $ (see (3.34) in \cite{FST08} for instance).

This verifies \eqref{eq-small-cap-and-Hu} and completes the proof.
\end{proof}

\begin{rem}
In the above proof, we estimated $\c_r$ by the so-called relative capacity of $\Sigma_r$ in $B_{2r} $.
By a result of Jauregui \cite{JJ21}, one can indeed check 
$$ \limsup_{r \to 0}  \left( 8 \pi - \int_{\Sigma_r} H | \nabla u | \right) \ge 0 .$$
\end{rem}

\begin{rem}
Alternatively one may use \eqref{eq-bdry-ineq-m-b} of Theorem \ref{thm-sec-ineq-2} (ii) to have
\bee
c_r \left( 4 \pi - \int_{\Sigma_r} | \nabla u |^2 \right) \le 4 \pi \, \m 
\eee
and check $  \int_{\Sigma_r} | \nabla u |^2  = O (1)$.
For instance, by the maximum principle, $ | \nabla u | \le  | \nabla v |$ at $ \p B_r$, where
$v $ is the harmonic function with $ v = 0 $ at $ \p B_r$ and  $ v = 1 $ at $ \p B_{2r}$.
By scaling and elliptic boundary estimates,
 $ \int_{\p B_r} | \nabla v |^2 = O (1)$ which shows $  \int_{\Sigma_r} | \nabla u |^2  = O (1)$.
\end{rem}

We want to mention that PMT is also an immediate corollary of 
Theorem \ref{thm-sec-more-ineq-3} in Section \ref{sec-promoting-ineq}, see Remark \ref{rem-one-more-proof-pmt}.

\section{Positive mass theorems with boundary} \label{sec-pmt-bdry}

Inspired by Proof \RomanNum{3} in the preceding section, 
we give some sufficient conditions that imply positive mass
on manifolds with boundary. 

\begin{thm} \label{thm-sec-pmt-bdry}
Let $(M, g)$ be a complete, orientable, asymptotically flat $3$-manifold 
with nonnegative scalar curvature, with boundary $\Sigma$.
Suppose $ \Sigma$ is connected and $H_2 (M, \Sigma) = 0$.
Let $ \Omega \subset M$ be a bounded region separating $\Sigma$ 
and $\infty$. More precisely, this means 
$ \p \Omega $ has two connected components $S_0$ and $S_1$, where 
$S_0$ encloses $\Sigma$ (and is allowed to 
coincide with $\Sigma$) and $S_1$ encloses $S_0$. 
Let $u_{_\Omega}$ be the function on $\Omega$ with 
$$ \Delta u_{_\Omega} = 0 , \ u_{_\Omega} |_{S_0} = 0 , \ \ \text{and} \ 
u_{_\Omega} |_{S_1 } = 1 . $$
Let 
$  \c (\Omega)  = \frac{1}{4\pi} \int_{\Omega} | \nabla u_{_\Omega}  |^2 
= \frac{1}{4\pi} \int_{S_0}  | \nabla u_{_\Omega} | $. 
Then
\be \label{eq-PM-condition-c}
H \le \frac{2}{ \c  (\Omega) }  \ \Longrightarrow \ \m > 0 . 
\ee
In particular, this implies 
\be \label{eq-PM-condition}
H \le  \frac{ 8 \pi L^2 }{\Vol(\Omega) } \  \Longrightarrow \ \m > 0 . 
\ee
Here $H$ is the mean curvature of $\Sigma $ in $(M,g)$, 
$ \Vol(\Omega)$ is  the volume of $(\Omega, g)$, and $L$ is the distance 
between $S_0$ and $S_1$. 
\end{thm}

\begin{proof} 
Let $u$ be the harmonic function on $M$ with $u=0$ at $\Sigma$ and $ u \to 1 $ at $\infty$.
By  \eqref{eq-bdry-ineq-m-a} of Theorem \ref{thm-sec-ineq-2} (i), 
\be \label{eq-mass-H-c}
\begin{split}
12 \pi \, \m \c_{_\Sigma}^{-1} \ge & \ 8 \pi - \int_{\Sigma} H | \nabla u | \\
\ge & \ 4 \pi ( 2 -  \c_{\Sigma}  \max_{\Sigma} H  ) ,
\end{split}
\ee
where $\c_{_\Sigma}$ is the capacity of $\Sigma$ in $(M, g)$. This shows 
\be \label{eq-pmt-ch}
 \max_{\Sigma} H \le 2 \c_{_\Sigma}^{-1}  \Longrightarrow  \m \ge  0 , 
\ \ \  \max_{\Sigma} H <  2 \c_{_\Sigma}^{-1}  \Longrightarrow  \m >  0 , 
\ee
respectively. 

Let $ D$ denote the region enclosed by $ S_1$ with $ \Sigma$. 
Extending $u_{_\Omega}$ to be $1$ on $M \setminus D$ and to be $0$ on 
$D \setminus \Omega$. 
By the variational characterization of the capacity,
\be \label{eq-c-and-com}
\c_{_\Sigma} < \frac{1}{4 \pi} \int_M | \nabla u_{_\Omega} |^2 =  \c (\Omega) .
\ee
Therefore, \eqref{eq-PM-condition-c} follows from \eqref{eq-pmt-ch} and \eqref{eq-c-and-com}.

To see \eqref{eq-PM-condition}, it suffices to estimate $ \c (\Omega)$. 
On $ \Omega$, consider a test function $f(x)$ which equals  $ L^{-1} d(x) $ if $ d (x) \le L $ and is identically $1 $ if $ d (x) \ge L$. Here $ d (x)$ denotes the distance from $x $ to $ S_0$.  
Then 
\be \label{eq-com-v-d}
\c (\Omega)  \le   \  \frac{1}{4\pi} \int_\Omega | \nabla f |^2 
\le \frac{1}{4\pi  L^2} \Vol (\Omega) .
\ee
Hence,  \eqref{eq-PM-condition} follows from \eqref{eq-PM-condition-c} and \eqref{eq-com-v-d}.
\end{proof}

The next result does not involve the mean curvature of the boundary. 
It makes use of \eqref{eq-bdry-ineq-m-b} in Theorem \ref{thm-sec-ineq-2} (ii).

\begin{thm} \label{thm-sec-pmt-bdry-2}
Let $(M, g)$, $\Omega$, $S_0$, $S_1$ and $u_{_\Omega}$ be given as in
Theorem \ref{thm-sec-pmt-bdry}.
Then
\be \label{eq-PM-condition-bu}
\int_{S_0} | \nabla u_{_\Omega} |^2 \le 4 \pi \Longrightarrow \ \m > 0 . 
\ee
\end{thm}

\begin{proof}
Let $ \tilde M$ denote the region outside $S_0$.  
Let $ \tilde u $ be the harmonic function on $ \tilde M$  with 
$ \tilde u =0$ at  $S_0$ and $ \tilde u \to 1 $ at $\infty$.
Applying \eqref{eq-bdry-ineq-m-b} of Theorem \ref{thm-sec-ineq-2} (ii) to  
$(\tilde M, g, \tilde u)$, we have
\be \label{eq-cap-band}
\int_{S_0} | \nabla \tilde u |^2 \le 4 \pi  \Longrightarrow \m \ge 0 , \ \ 
\int_{S_0} | \nabla \tilde u |^2 < 4 \pi  \Longrightarrow \m > 0 , 
\ee
respectively. On $(\Omega, g)$, the maximum principle shows
\be \label{eq-grad-u-comparison}
| \nabla \tilde u  | <  | \nabla u_{_\Omega} | \ \text{at} \ S_0. 
\ee
Therefore, \eqref{eq-PM-condition-bu} follows from \eqref{eq-cap-band} and \eqref{eq-grad-u-comparison}. 
\end{proof}

\begin{rem}
It may be worthy of noting that the condition in  \eqref{eq-PM-condition-bu}
and the upper bound of $H$ in \eqref{eq-PM-condition-c} 
only involve the $C^0$-geometry of $(\Omega, g)$.
\end{rem}


It is conceivable that Theorems \ref{thm-sec-pmt-bdry} and \ref{thm-sec-pmt-bdry-2}
may be used to study the mass of incomplete asymptotically flat $3$-manifolds. 
Recently Cecchini-Zeidler \cite{CZ21} and 
Lee-Lesourd-Unger \cite{LLU22} have given sufficient conditions, 
involving a positive lower bound of the scalar curvature 
on suitable regions in a manifold $(M^n, g)$ that is spin or of dimension $ 3 \le n \le 7$,
which guarantee the positivity of the mass. 
If such conditions are interpreted as shielding the incomplete part by regions with sufficiently 
positive scalar curvature, conditions in \eqref{eq-PM-condition-c}, \eqref{eq-PM-condition}
and \eqref{eq-PM-condition-bu} may be thought as shielding conditions in terms of the $C^0$-geometry of a separating region.

\vspace{.2cm}

We end this section with the following proposition which was known and proved previously via the weak inverse mean curvature ($1/H$) flow developed by Huisken-Ilmanen \cite{HI01}.
We include it here to show that the result can also be proved using harmonic functions.

\begin{prop} \label{prop-Willmore-pmt}
Let $(M, g)$ be a complete, orientable, asymptotically flat $3$-manifold with boundary $\Sigma$. Suppose $ \Sigma$ is connected and $H_2 (M, \Sigma) = 0$. If $g$ has nonnegative scalar curvature, 
then 
$$ \int_{\Sigma} H^2 \le 16 \pi \  \Longrightarrow \  \m \ge 0  , $$
and $\m = 0 $ if and only if $(M,g)$ is isometric to $ \R^3$ minus a round ball.
\end{prop}

\begin{proof}
By Corollary \ref{cor-sec-ineq}, 
$$  \int_{\Sigma} H^2 \le 16 \pi  \ \Longrightarrow \ \int_\Sigma | \nabla u |^2 \le 4 \pi . $$
Hence,  $ \m \ge 0 $ by \eqref{eq-bdry-ineq-m-b}. The rigidity case follows from that of Corollary \ref{cor-sec-ineq}.
\end{proof}

\section{Integral identities for the mass-to-capacity ratio} \label{sec-mass-integrals}

In \cite{BKKS19}, Bray-Kazaras-Khuri-Stern found an integral identity for the mass of an asymptotically flat 
manifold. More precisely, if $(E, g)$ denotes the exterior region of a complete, asymptotically flat 
Riemannian $3$-manifold $(M, g)$ with mass $\m$, then
\be \label{eq-BKKS}
16 \pi \m \ge \int_{E} \left( \frac{ | \nabla^2 u  |}{  | \nabla u | } + R | \nabla u | \right) ,
\ee
where $u$ is a harmonic function on $ (E, g)$ satisfying Neumann boundary conditions at $ \p E$,
and which is asymptotic to one of the asymptotically flat coordinate functions at $\infty$. 
In particular, if the scalar curvature is nonnegative, then $\m \ge 0 $.

In this section, we derive mass identities analogous to \eqref{eq-BKKS} with $u$ being 
a harmonic function that equals $0$ at the boundary and is asymptotic to $1$ at $\infty$.

\begin{thm} \label{thm-sec-mass-integrals}
Let $(M, g)$ be a complete, orientable, asymptotically flat $3$-manifold with boundary $\Sigma$.
Suppose $ \Sigma$ is connected and $H_2 (M, \Sigma) = 0$.
Let $ u $ be the harmonic function such that $ u = 0 $ at $ \Sigma$ and $ u \to 1$ at $\infty$. 
Let $ \Phi_u$ be a symmetric $(0,2)$ tensor given by
\bee
\begin{split}
\Phi_u  =  \frac{ | \nabla u |^2 } { 1 - u }  g  -  \frac{  3 d u \otimes d u }{1-u}  .
\end{split}
\eee
Let $ \m $ be the mass of $(M, g)$ and $ \c_{_\Sigma}$ be the capacity of $\Sigma$ in $(M, g)$.
Then
\be \label{eq-M-identity}
\begin{split}
& \  \m \c_{_\Sigma}^{-1} - \left( 1  - \frac{1}{ 4 \pi}  \int_{\Sigma} | \nabla u |^2 \right) \\
\ge & \  \frac{1}{ 16 \pi} \int_M \left[ \frac{1}{ (1- u )^2} - 1 \right] \left(  \frac{ | \nabla^2 u - \Phi_u |^2 } {  | \nabla u | } 
+  R | \nabla u |  \right) 
\end{split} 
\ee
and
\be \label{eq-M-identity-2}
\begin{split}
& \  \m \c_{_\Sigma}^{-1} - \frac23 \left( 1 - \frac{1}{8\pi}  \int_{\Sigma} H | \nabla u | \right)  \\
\ge & \  \frac{1}{ 16 \pi} \int_M  \left[ \frac{1}{  (1- u )^2} - \frac13 \right] 
\left(  \frac{ | \nabla^2 u - \Phi_u |^2 } {  | \nabla u | } 
+  R | \nabla u |  \right) .
\end{split} 
\ee
\end{thm}

\begin{proof}
By \eqref{eq-lap-conse}, along a regular level set $\Sigma_t$, $ (\nabla^2 u - \Phi_u )$ satisfies 
$$
\left( \nabla^2 u -  \Phi_u \right)  (\nu, \nu) =   - H | \nabla u |  + \frac{ 2 | \nabla u |^2}{1-u} ,  \ 
$$
$$  \left( \nabla^2 u -  \Phi_u \right)  (\nu, \cdot) |_{_{\Sigma_t} } = \la \nabla_{_{\Sigma_t}} | \nabla u | , \cdot \ra ,  $$
 $$
\left( \nabla^2 u -  \Phi_u  \right)  ( \cdot, \cdot) |_{_{\Sigma_t} } 
= | \nabla u | \left(  \Pi -  \frac{ | \nabla u | } { 1 - u }   \gamma \right) , 
$$
where $\gamma$ denotes the induced metric on $ \Sigma_t$. Therefore, 
\be \label{eq-difference-hessian-Phiu}
\begin{split}
& \  | \nabla u |^{-2}  \left| \nabla^2 u - \Phi_u  \right|^2  \\
= & \ \frac32 \left( H - \frac{ 2 | \nabla u |}{ 1 - u } \right)^2 + 2 | \nabla u |^{-2} | \nabla_{_{\Sigma_t} } | \nabla u | |^2  + | \mathring{\Pi} |^2 .
\end{split}
\ee

Given two regular values $t_1 < t_2$, by \eqref{eq-difference-mathcal-B} in Proposition \ref{prop-reg-monotone} of the Appendix, we have
\be  
\label{eq-d-B-r}
\mathcal{B} (t_2) - \mathcal{B} (t_1) 
=   \int_{t_1}^{t_2}   \frac{1}{(1-t)^2}  \left[ 3 B(t) - A (t) \right] .
\ee
By \eqref{eq-Psi-B-A} and \eqref{eq-Psi-expression}, 
\be 
3 B(t) - A (t) = \frac{1}{1-t} \Psi (t) \ \ \text{and} \ \ 
\Psi (t) \ge \int_{t}^{1} \psi(s) , 
\ee
where
\be 
\begin{split}
\psi (t) = & \ \int_{\Sigma_t}  \left[  \frac34 \left( H - \frac{ 2 | \nabla u | }{ 1 - u } \right)^2 
+ | \nabla u |^{-2} | \nabla_{_{\Sigma_t} } | \nabla u | |^2 
+ \frac12 | \mathring{ \Pi } |^2 +  \frac12 R  \right] \\
= & \ \frac12  \int_{\Sigma_t}  \left( \frac{ | \nabla u  - \Phi_u  |^2 }{ | \nabla u |^{2} }  +  R  \right) . 
\end{split}
\ee

Taking $ t_1 = 0 $ and letting  $t_2 \to 1$,  applying Theorem \ref{thm-limits}, we hence have
\be  \label{eq-mass-relation-1}
\begin{split}
4 \pi \m \c_{_\Sigma}^{-1}  - \mathcal{B} ( 0 ) 
= & \  \int_{0}^{1}   \frac{1}{(1-t)^3} \Psi (t) \\
\ge & \ \int_{0}^{1}   \frac{1}{(1-t)^3}  \left( \int_t^1 \psi (s) \right) .
\end{split}
\ee
Integration by parts gives 
\be
\begin{split}
& \ \int_{0}^{1}   \frac{1}{(1-t)^3}  \left( \int_t^1 \psi (s) \right) \\
= & \ \frac12 \left[ \lim_{t \to 1}  \frac{1}{  (1-t)^2 }     \int_t^1 \psi (s)  
- \int_0^1  \psi (s) +   \int_{0}^{1} \frac{ \psi (t) }{  (1-t)^2} \right] .
\end{split}
\ee
We claim
\be \label{eq-end-tend-to-zero}
 \lim_{t \to 1}  \frac{1}{  (1-t)^2 }     \int_t^1 \psi (s)   = 0 . 
\ee
This is because, by \eqref{eq-int-error-term},
$$
  \int_{t}^1 \int_{\Sigma_s}  | \nabla u |^{-2} | \nabla_{_\Sigma} | \nabla u | |^2 +  \frac{1}{2} | \mathring{\Pi} |^2   
  = O ( ( 1 - t )^{1 + 2 \tau} ), 
$$
and, by \eqref{eq-int-scalar}, 
$$
 \int_t^1 \int_{\Sigma_s}  |R| = o ( ( 1 - t )^2 ).
$$
Also, by \eqref{eq-H-and-t},  \eqref{eq-nabu-and-t} and Lemma \ref{lem-area}, 
\be
\int_t^1 \int_{\Sigma_s} \left( H - \frac{2 | \nabla u |}{ 1 - u } \right)^2 = O ( ( 1 - t)^{1 + 2 \tau} ) .
\ee
Therefore, \eqref{eq-end-tend-to-zero} holds. 

Now it follows from \eqref{eq-mass-relation-1} -- \eqref{eq-end-tend-to-zero} that
\be \label{eq-M-B-0}
\begin{split}
& \   4 \pi \m \c_{_\Sigma}^{-1} - \mathcal{B} (0)  \\
\ge & \ \frac12 \int_0^1 \left[ \frac{1}{ (1-t)^2} - 1 \right] \psi (t) \\
= & \ \frac14  \int_0^1 \left[ \frac{1}{ (1-t)^2} - 1 \right] 
 \int_{\Sigma_t}  \left(  \frac{ | \nabla^2 u - \Phi_u |^2 }{  | \nabla u |^{2} } +  R  \right) \\
 = & \ \frac14 \int_M \left[ \frac{1}{ (1- u )^2} - 1 \right] 
 \left( \frac{ | \nabla^2 u - \Phi_u |^2 } {| \nabla u | }  +  R | \nabla u |  \right).
\end{split}
\ee
This proves \eqref{eq-M-identity}.

Similarly, by \eqref{eq-monotone-AMO} in Proposition \ref{prop-reg-monotone} of the Appendix,
\be \label{eq-AMO-difference} 
\left[ \mathcal{A}(t_2) - \mathcal{B}(t_2) \right] - \left[ \mathcal{A}(t_1) - \mathcal{B}(t_1) \right]  
\ge \int_{t_1}^{t_2}   \frac{1}{(1-t)^2}  \psi (t) .
\ee
Taking $ t_1 = 0 $, letting $t_2 \to 1$ and applying Theorem \ref{thm-limits}, we have
\be \label{eq-AMO-M-identity}
\begin{split}
& \ 8 \pi \m \c_{_\Sigma}^{-1}  - (\mathcal{A}(0) - \mathcal{B}(0) ) \\
\ge & \ \frac12 \int_{0}^{1}   \frac{1}{(1-t)^2}  
\int_{\Sigma_t} \left(  \frac{ | \nabla^2 u - \Phi_u |^2 }{  | \nabla u |^{2} } +  R  \right)  \\ 
= & \  \frac12  \int_M   \frac{1}{(1-u)^2}  
\left( \frac{ | \nabla^2 u - \Phi_u |^2 }{  | \nabla u | }  +  R | \nabla u |  \right) .
\end{split}
\ee
This together with \eqref{eq-M-B-0} proves \eqref{eq-M-identity-2}.
\end{proof}

\begin{rem} \label{rem-AMO-ineq}
If the scalar curvature $ R$ is nonnegative,
then \eqref{eq-M-identity} implies \eqref{eq-bdry-ineq-m-b}, 
\eqref{eq-M-identity-2} implies \eqref{eq-bdry-ineq-m-a}, and
\eqref{eq-AMO-M-identity} implies 
\be \label{eq-AMO-ineq}
4 \pi - \int_{\Sigma} H | \nabla u | + \int_{\Sigma} | \nabla u |^2 \le 8 \pi \m \c_{_\Sigma}^{-1} .
\ee
For manifolds that are spatial Schwarzschild manifolds near infinity, 
\eqref{eq-AMO-ineq} also followed from the work of  Agostiniani-Mazzieri-Oronzio \cite{AMO21}.
On the other hand, one can see 
\eqref{eq-AMO-ineq} is an algebraic consequence of  
\eqref{eq-bdry-ineq} and
\eqref{eq-bdry-ineq-m-b}. 
\end{rem}

\begin{rem}
Suppose $M$ in Theorem \ref{thm-sec-mass-integrals} has no boundary. 
Taking $ u = 1 - 4 \pi G$, 
where $G$ is the minimal  positive Green's function with a pole at some $ p \in M$, 
and letting $t_2 \to 1 $ and $ t_1 \to - \infty$ in \eqref{eq-AMO-difference}, 
one has 
\be \label{eq-AMO-M-g-identity}
\begin{split}
 \m  
\ge & \  \frac{1}{ ( 8 \pi )^2 }  \int_M   \frac{1}{ G^2}  
\left( \frac{ |   \nabla^2 G  +   \Phi_G  |^2 }{    | \nabla G  | }  +    R  | \nabla G |  \right)  .
\end{split}
\ee
Here $ \Phi_G$ is the $(0,2)$ tensor given by
\bee
\begin{split}
\Phi_G  =    \frac{ | \nabla G  |^2 } {  G }  g  -  \frac{  3 d G \otimes d  G  }{ G }  .
\end{split}
\eee 
\eqref{eq-AMO-M-g-identity} gives the integral version of the proof of 
the $3$-d positive mass theorem in \cite{AMO21}.

\end{rem}

\section{Promoting inequalities via Schwarzschild models} \label{sec-promoting-ineq}

Inequalities in Section \ref{sec-ineq} are derived via monotone quantities that become constant  
in Euclidean spaces outside round balls.
As a result, they are strict inequalities when evaluated in spatial Schwarzschild manifolds with nonzero mass
outside rotationally symmetric spheres.

Inspired by Bray's proof of the Riemannian Penrose inequality \cite{Bray02}, in this section we apply results 
from the previous sections to derive inequalities that become equality in Schwarzschild spaces.

We first outline the idea. Given a tuple $(M, g, u)$ satisfying assumptions in 
Theorem \ref{thm-sec-ineq} (or equivalently in Theorem \ref{thm-sec-ineq-2}), 
let $v$ be any other harmonic function on $(M, g)$ with $ v \to 1 $ at $\infty$ and $ v > 0 $ at $ \Sigma$. 
The following facts hold:

\vspace{.1cm}

\begin{enumerate}
\item the metric $ \bar g : = v^4 g $ is asymptotically flat, with nonnegative scalar curvature; 

\vspace{.1cm}

\item the function $ \bar u : = v^{-1} u $ is a harmonic function with respect to the metric $\bar g$,
and satisfies $ \bar u = 0 $ at $ \Sigma$ and $ \bar u \to 1 $ at $ \infty $. 

\end{enumerate}
Thus, results from the previous sections are applicable to $M$ with the conformally deformed metric $\bar g$
and the $\bar g$-harmonic function $\bar u$.

To proceed, we compute the quantities involved.  
Let $\bar \nabla$  denote the gradient on $(M, \bar g)$, let $\bar H$ be the mean curvature of $ \Sigma$
in $(M, \bar g)$ with respect to the $\infty$-pointing normal. 
Let $ d \sigma$, $ d \bar \sigma$ denote the surface measure on $\Sigma$ in $(M, g)$, $(M, \bar g)$, respectively. 
As $\Sigma$ has dimension two, it can be checked 
\be \label{eq-conformal-bdry-relation-1}
\int_{\Sigma} | \bar \nabla \bar u |_{\bar g}^2 \, d \bar \sigma = \int_{\Sigma}  | \nabla \bar u |^2 \, d \sigma .
\ee
(We omitted writing the area and volume measures in previous integrals as there 
was only one metric $g$ involved therein.)
The mean curvature $\bar H$ is related to the mean curvature $H$ of $ \Sigma$ in $(M, g)$ via 
$ \bar H = v^{-2} \left( 4 v^{-1} \p_\nu v + H \right) $. 
Thus,
\be
\int_\Sigma \bar H | \bar \nabla \bar u |_{\bar g} \, d \bar \sigma = 
\int_\Sigma   \left( 4 v^{-1} \p_\nu v + H \right)    |  \nabla \bar u | \,  d \sigma .
\ee
Let $\bar \m $ denote the mass of $(M, \bar g)$. $\bar \m $ and $\m$ are related by 
\be
\bar \m = \m - 2 \c_v ,
\ee
where $\c_v$ is the constant in the expansion 
\bee
v = 1 - \frac{ \c_v}{ |x| } + o (|x|^{-1}) ,
\eee
as $ x \to \infty$. Since $\bar u = v^{-1} u $,  $\bar u $ satisfies
\bee
\bar u =  1 - \frac{ \left( \c_{_\Sigma} - \c_v \right)  }{ | x|} + o ( |x|^{-1} ),
\eee
where $ \c_{\Sigma } > \c_v$ by the fact $ v > u $ and the maximum principle.
The capacity of $\Sigma$ in $(M, \bar g)$, which we denote by $ \bar \c_{_\Sigma}$, is then given by
\be
\bar \c_{_\Sigma} = \c_{_\Sigma} - \c_v . 
\ee
Finally, we note, as $ u = 0 $ at $ \Sigma$,
\be \label{eq-conformal-bdry-relation-l}
| \nabla \bar u |  = v^{-1} | \nabla u | \ \ \text{at} \ \Sigma. 
\ee

We want to seek implications of the 
inequalities \eqref{eq-bdry-ineq}, \eqref{eq-bdry-ineq-m-b},  \eqref{eq-bdry-ineq-m-a}  and
\eqref{eq-AMO-ineq}, i.e. 
\be \label{eq-ineq-repeat-1}
4 \pi + \int_{\Sigma} H | \nabla u | \ge 3 \int_{\Sigma} | \nabla u |^2 ,
\ee
\be \label{eq-ineq-repeat-2}
4 \pi - \int_{\Sigma} | \nabla u |^2 \le 4 \pi \m \, \c_{_\Sigma}^{-1}, 
\ee
\be \label{eq-ineq-repeat-3}
8 \pi - \int_{\Sigma} H | \nabla u |  \le 12 \pi \m \, \c_{_\Sigma}^{-1}, 
\ee
\be \label{eq-ineq-repeat-4}
4 \pi - \int_{\Sigma} H | \nabla u | +  \int_{\Sigma} | \nabla u |^2 \le 8 \pi \m \, \c_{_\Sigma}^{-1} ,
\ee
when they are applied to the conformally deformed triple $(M, \bar g, \bar u )$.
As mentioned in Remark \ref{rem-comaring} and Remark \ref{eq-AMO-ineq}, 
one knows 
$$ \eqref{eq-ineq-repeat-1}  \  + \  \eqref{eq-ineq-repeat-2} \ \Longrightarrow \ 
\eqref{eq-ineq-repeat-3} \ \ \text{and}  \ \ \eqref{eq-ineq-repeat-4} . $$
For this reason, we focus on the use of \eqref{eq-ineq-repeat-1} and \eqref{eq-ineq-repeat-2} below.

\begin{thm} \label{thm-sec-conformal-ineq-1}
Let $(M, g)$ be a complete, orientable, asymptotically flat $3$-manifold with boundary $\Sigma$.
Suppose $ \Sigma$ is connected and $H_2 (M, \Sigma) = 0$.
Let $ u $ be the harmonic function such that $ u = 0 $ at $ \Sigma$ and $ u \to 1$ at $\infty$. 
If $ g$ has nonnegative scalar curvature, then, for any constant $ k > 0$, 
\be \label{eq-sec-conformal-ineq-1}
4 \pi + k \int_\Sigma   H  |  \nabla u | \ge k ( 4 - k ) \int_\Sigma  |  \nabla u |^2.
\ee
Moreover, equality in \eqref{eq-sec-conformal-ineq-1} holds for some $k$ 
if and only if $(M, g)$ is isometric to a spatial Schwarzschild manifold outside a rotationally symmetric sphere,
that is, up to isometry, 
$$ (M, g) = \left( \R^3 \setminus \{ | x | < r \} , \left( 1 + \frac{\m}{ 2 | x |} \right)^4 g_{_E} \right) ,$$
where $ r > 0 $ is a constant,  $ g_{_E} = \delta_{ij} d x^i d x^j$ is the Euclidean metric,  and $\m$, $k$, $r$ 
are related by $ \m = 2 r ( k - 1 ) $.

\end{thm}
\begin{proof}
Given any positive harmonic function $ v  $ on $(M, g)$, let $ \bar g = v^4 g $ and $ \bar u = v^{-1} u $.
Applying \eqref{eq-bdry-ineq} in Theorem \ref{thm-sec-ineq} to the triple $(M, \bar g, \bar u )$, we have
\be \label{eq-conformal-ineq-1}
4 \pi + \int_\Sigma \bar H | \bar \nabla \bar u |_{\bar g } \, d \bar \sigma  
\ge  3 \int_\Sigma | \bar \nabla \bar u |_{\bar g}^2  \, d \bar \sigma .
\ee
By \eqref{eq-conformal-bdry-relation-1} -- \eqref{eq-conformal-bdry-relation-l}, \eqref{eq-conformal-ineq-1} shows
\be \label{eq-conformal-ineq-2}
4 \pi + \int_\Sigma   \left( 4 v^{-1} \p_\nu v + H \right)  v^{-1}  |  \nabla u | \,  d \sigma  
\ge  3 \int_\Sigma  v^{-2} |  \nabla u |^2  \, d \sigma . 
\ee

Given any constant $k > 0$, choose 
\be \label{eq-choice-of-v-k}
  v = u + \frac{1}{k} ( 1 - u ) . 
\ee
It follows from \eqref{eq-conformal-ineq-2} and the fact $ \p_\nu u = | \nabla u | $ at $ \Sigma$ that
\bee
4 \pi + k \int_\Sigma   H  |  \nabla u | \,  d \sigma  
\ge  k ( 4 - k ) \int_\Sigma  |  \nabla u |^2  \, d \sigma ,
\eee
which proves \eqref{eq-sec-conformal-ineq-1}.

The above also shows equality in \eqref{eq-sec-conformal-ineq-1} holds for some $ k $ if and only if 
equality in \eqref{eq-conformal-ineq-1} holds for the corresponding $(M, \bar g, \bar u )$.
By Theorem \ref{thm-sec-ineq}, this occurs if and only if $(M, \bar g)$ is isometric 
to $ \left( \R^3 \setminus B_r , g_{_E} \right)$, where $B_r = \{ x \in \R^3 \ | \ |x| < r \}$ 
for some constant $ r > 0$.
In this case, 
\be
\bar u = 1 - \frac{r}{ | x| } .
\ee
This combined with \eqref{eq-choice-of-v-k} and the fact $ \bar u = v^{-1} u $ shows
\be
v^{-1} = 1 + \frac{ r (k-1)}{ |x| } .
\ee
As a result, 
\be
g = v^{-4} g_{_E} = \left( 1 + \frac{ r (k-1)}{ |x| } \right)^4 \delta_{ij} \, d x^i d x^j ,
\ee
which is a spatial Schwarzschild metric with mass $ \m = 2 r ( k -1 ) $.
\end{proof}

Theorem \ref{thm-sec-conformal-ineq-1} implies a sharp bound of $\int_\Sigma | \nabla u |^2$ by the Willmore functional
of $\Sigma$, with the bound achieved by Schwarzschild spaces outside mean-convex round spheres. 

\begin{cor} \label{cor-sec-more-improved-B-M}
Let $(M, g)$ be a complete, orientable, asymptotically flat $3$-manifold with boundary $\Sigma$.
Suppose $ \Sigma$ is connected and $H_2 (M, \Sigma) = 0$.
Let $ u $ be the harmonic function such that $ u = 0 $ at $ \Sigma$ and $ u \to 1$ at $\infty$. 
If $ g$ has nonnegative scalar curvature, then 
\be \label{eq-improved-B-M} 
 \left( \frac{1}{\pi} \int_\Sigma  |  \nabla u |^2 \right)^\frac12  \le  \left( \frac{1}{16\pi}  \int_\Sigma   H^2 \right)^\frac12  + 1 .
\ee
Moreover,  equality holds if and only if 
 $(M, g)$ is isometric to a spatial Schwarzschild manifold  outside a rotationally symmetric sphere
 with nonnegative constant mean curvature. 
 \end{cor}

\begin{proof}
Consider the following quadratic form of $ k$,
\be
Q (k) : =  \alpha (u) \, k^2  + \beta (u) \, k  +  4 \pi  ,
\ee
where
$$ \alpha (u) =  \int_\Sigma  |  \nabla u |^2 , \ \ 
\beta (u) =  \int_\Sigma   H  |  \nabla u |  - 4 \int_\Sigma  |  \nabla u |^2  . $$ 
 We have $ Q(0) = 4 \pi $, and Theorem \ref{thm-sec-conformal-ineq-1} shows
$$ Q( k ) \ge 0, \ \forall \ k > 0 . $$
Thus, by elementary reasons, either 
\be \label{eq-discriminant}
 \beta (u)^2 - 16 \pi \alpha (u) \le 0 
\ee
or
\be \label{eq-condition-au-bu-2}
\beta (u)^2 - 16 \pi \alpha (u) > 0  \ \text{and} \ 
- \beta (u) + \sqrt{ \beta (u)^2 - 16 \pi \alpha (u)  } < 0 . 
\ee
The latter case is equivalent to
$$
\beta (u) > \sqrt{ 16 \pi \alpha (u) } ,
$$
that is 
\be \label{eq-interesting-ineq-1}
\frac{  \int_\Sigma   H  |  \nabla u |  }{ \left( \int_\Sigma  |  \nabla u |^2 \right)^\frac12  } 
 - 4  \left( \int_\Sigma  |  \nabla u |^2 \right)^\frac12  > \sqrt{ 1 6 \pi  } .
\ee
If \eqref{eq-interesting-ineq-1} holds, then, by H\"{o}lder's inequality,
\be \label{eq-interesting-ineq-2}
\left( \frac{1}{16\pi}  \int_\Sigma   H^2 \right)^\frac12  >  1  +   \left( \frac{1}{\pi}  \int_\Sigma  |  \nabla u |^2 \right)^\frac12  .
\ee
If  \eqref{eq-discriminant} holds, then 
\bee 
\left| \int_\Sigma   H  |  \nabla u |   - 4 \int_\Sigma  |  \nabla u |^2  \right| \le 4 \left(  \pi \int_\Sigma  |  \nabla u |^2  \right)^\frac12 ,
\eee
which in particular implies 
\be \label{eq-interesting-ineq-3}
4 \int_\Sigma  |  \nabla u |^2   \le 4 \left(  \pi \int_\Sigma  |  \nabla u |^2  \right)^\frac12 +  \int_\Sigma   H  |  \nabla u | .
\ee
Combined with  H\"{o}lder's inequality, this shows 
\be  \label{eq-interesting-ineq-4}
 \left( \frac{1}{\pi} \int_\Sigma  |  \nabla u |^2 \right)^\frac12  \le 1  +   \left( \frac{1}{16\pi}  \int_\Sigma   H^2 \right)^\frac12 .
\ee
Therefore, in either case, we conclude \eqref{eq-improved-B-M} holds. 

If equality in \eqref{eq-improved-B-M} holds, then \eqref{eq-condition-au-bu-2} does not hold;
\eqref{eq-interesting-ineq-3} holds with equality; and $ H = c | \nabla u | $ for some constant $ c \ge 0 $.
In particular, this gives 
$$  - \beta (u) =  4 \int_\Sigma  |  \nabla u |^2  -  \int_\Sigma   H  |  \nabla u |  = \sqrt{ 16 \pi \alpha (u) } > 0 .  $$
As a result, $ Q (k_0) = 0 $ at 
$$ k_0 = - \frac{\beta (u)}{2 \alpha (u) } = 2 \left( \frac{1}{\pi} \int_\Sigma | \nabla u |^2  \right)^{-\frac12}  
= \frac{2} { 1 + \left( \frac{1}{16 \pi} \int_{\Sigma} H^2  \right)^\frac12 }  > 0 . $$
By Theorem \ref{thm-sec-conformal-ineq-1}, 
$(M, g)$ is isometric to a spatial Schwarzschild manifold
$$ (M, g) = \left( \R^3 \setminus \{ | x | < r \} , \left( 1 + \frac{\m}{ 2 | x |} \right)^4 g_{_E} \right) ,$$
where $ r > 0 $ and $ \m = 2 r ( k_0 - 1 ) $. 
As $k_0 \le 2$, the boundary $\{ |x| = r \} $ has nonnegative mean curvature in $(M, g)$.

On such an $(M, g)$, direct calculation shows 
$$  \left( \frac{1}{16 \pi} \int_\Sigma H^2 \right)^\frac12  = \left| \frac{2}{k} - 1 \right| 
\ \ \text{and} \ \ \left( \frac{1}{ \pi} \int_{\Sigma} | \nabla u |^2 \right)^\frac12 = \frac{2}{k} . 
$$
As $ k \le 2$, equality holds in \eqref{eq-improved-B-M}.
This completes the proof.
\end{proof}

An immediate application of Corollary  \ref{cor-sec-more-improved-B-M} yields a result of 
Bray and the author \cite{BrayMiao08} on the estimate of the capacity-to-area-radius ratio. 

\begin{thm}[\cite{BrayMiao08}] \label{thm-BM07}
Let $(M, g)$ be a complete, orientable, asymptotically flat $3$-manifold with boundary $\Sigma$.
Suppose $ \Sigma$ is connected and $H_2 (M, \Sigma) = 0$. 
If $g$ has nonnegative scalar curvature, then
\be \label{eq-Bray-Miao-06} 
 \frac{  2  \c_{_\Sigma} } {  r_{_\Sigma}  }  \le  \left( \frac{1}{16\pi}  \int_\Sigma   H^2 \right)^\frac12  + 1 .
\ee
Here $\c_{_\Sigma}$ is the capacity of $\Sigma$ in $(M, g)$ and 
$ r_{_\Sigma} = \left( \frac{ | \Sigma | }{4 \pi}  \right)^\frac12 $ is the area-radius of $\Sigma$.
Moreover, equality holds if and only if 
$(M, g)$ is isometric to a spatial Schwarzschild manifold outside a rotationally symmetric sphere 
with nonnegative constant mean curvature. 
 \end{thm}

\begin{proof}
This follows directly from  
$$ 
 \left( \int_\Sigma  |  \nabla u |^2 \right)^\frac12 \ge \frac{ \int_{\Sigma} | \nabla u | } { | \Sigma |^\frac12 } 
 =   \sqrt{\pi } \frac{ 2  \c_{_\Sigma} }{  r_{_\Sigma} } 
$$
and Corollary \ref{cor-sec-more-improved-B-M}.
\end{proof}

Next, we proceed to find implications of \eqref{eq-bdry-ineq-m-b} in Theorem \ref{thm-sec-ineq-2}.

\begin{thm} \label{thm-sec-more-ineq-1}
Let $(M, g)$ be a complete, orientable, asymptotically flat $3$-manifold with boundary $\Sigma$.
Suppose $ \Sigma$ is connected and $H_2 (M, \Sigma) = 0$.
Let $ u $ be the harmonic function such that $ u = 0 $ at $ \Sigma$ and $ u \to 1$ at $\infty$. 
If $ g$ has nonnegative scalar curvature, then
\be \label{eq-conformal-ineq-m-b-main}
\begin{split}
 \frac{ \m} { 2  \c_{_\Sigma} }   \ge 1  -  \left(  \frac{1}{4 \pi}  \int_{\Sigma} | \nabla u |^2 \, d \sigma \right)^\frac12  .
\end{split}
\ee
Moreover, equality holds if and only if 
$(M, g)$ is isometric to a spatial Schwarzschild manifold outside a rotationally symmetric sphere.
\end{thm}

\begin{proof}
Given any positive harmonic function $ v  $ on $(M, g)$, let $ \bar g = v^4 g $ and $ \bar u = v^{-1} u $.
Applying \eqref{eq-bdry-ineq-m-b} in Theorem \ref{thm-sec-ineq-2} to $(M, \bar g, \bar u )$, we have
\be \label{eq-conformal-bdry-ineq-2}
4 \pi - \int_{\Sigma} | \bar \nabla \bar u |^2_{\bar g} \, d \bar \sigma \le 4 \pi  \bar \m { \bar \c_{_\Sigma} }^{-1} .
\ee
By \eqref{eq-conformal-bdry-relation-1} -- \eqref{eq-conformal-bdry-relation-l}, \eqref{eq-conformal-bdry-ineq-2} becomes 
\be \label{eq-conformal-ineq-m-b}
4 \pi - \int_{\Sigma}  v^{-2} | \nabla u |^2 \, d \sigma  \le 4 \pi  \frac{ \m - 2 \c_v} { \c_{_\Sigma} - \c_v } .
\ee

Given any constant $ k > 0 $, choose  
\be \label{eq-v-and-u}
 v = u + \frac{1}{k} ( 1 - u ) . 
\ee
Then $ v = k^{-1} $ at $ \Sigma$, $ \c_{v} = (1 - k^{-1} ) \c_{_\Sigma} $, 
and \eqref{eq-conformal-ineq-m-b} shows
\be \label{eq-conformal-ineq-m-b-1}
\begin{split}
 \frac{ \m} {  \c_{_\Sigma} }   \ge 2  -  \frac{1}{k}  -   \frac{k}{4 \pi}  \int_{\Sigma} | \nabla u |^2 \, d \sigma .
\end{split}
\ee
Maximizing the right side of \eqref{eq-conformal-ineq-m-b-1} over all $ k > 0 $, we have
\be \label{eq-conformal-ineq-m-b-2}
\begin{split}
 \frac{ \m} { 2  \c_{_\Sigma} }   \ge 1  -  \left(  \frac{1}{4 \pi}  \int_{\Sigma} | \nabla u |^2 \, d \sigma \right)^\frac12  ,
\end{split}
\ee
which proves \eqref{eq-conformal-ineq-m-b-main}.

If equality in \eqref{eq-conformal-ineq-m-b-main} holds, then 
equality in \eqref{eq-conformal-bdry-ineq-2} holds
for $ v = u + k^{-1} ( 1 - u ) $ with the constant $k $ given by 
\be
 k = \left(  \frac{1}{4 \pi}  \int_{\Sigma} | \nabla u |^2 \, d \sigma \right)^{ - \frac12 } .    
\ee
By Theorem \ref{thm-sec-ineq-2},  $(M, \bar g)$ is isometric 
to $ \left( \R^3 \setminus B_r , g_{_E} \right)$, where $B_r = \{ x \in \R^3 \ | \ |x| < r \}$ for some $ r > 0$,  
and
\be
\bar u = 1 - \frac{r}{ | x| } .
\ee
This combined with $\bar u = v^{-1} u $ and \eqref{eq-v-and-u} shows
\be
v^{-1} = 1 + \frac{ r (k-1)}{ |x| } .
\ee
As a result, 
\be
g = v^{-4} g_{_E} = \left( 1 + \frac{ r (k-1)}{ |x| } \right)^4 \delta_{ij} \, d x^i d x^j ,
\ee
which is a spatial Schwarzschild metric with the mass $\m = 2 r ( k -1 )$.

On any such an $(M, g)$, direct calculation shows 
$$  \frac{\m}{2 \c_{_\Sigma} }   = 1 - \frac{1}{k} 
\ \ \text{and} \ \ \left( \frac{1}{ 4 \pi} \int_{\Sigma} | \nabla u |^2 \right)^\frac12 = \frac{1}{k} ,
$$
which verifies equality in \eqref{eq-conformal-ineq-m-b-main}. This completes the proof.
\end{proof}

We now have a succinct lower bound of the mass-to-capacity ratio by the Willmore functional.

\begin{thm} \label{thm-sec-more-ineq-3}
Let $(M, g)$ be a complete, orientable, asymptotically flat $3$-manifold with boundary $\Sigma$.
Suppose $ \Sigma$ is connected and $H_2 (M, \Sigma) = 0$.
If $ g$ has nonnegative scalar curvature, then
\be \label{eq-conformal-ineq-m-b-main-3}
\begin{split}
 \frac{ \m} {  \c_{_\Sigma} }   \ge 1  -  \left( \frac{1}{16\pi}  \int_\Sigma   H^2 \right)^\frac12  .
\end{split}
\ee
Moreover, equality holds if and only if 
$(M, g)$ is isometric to a spatial Schwarzschild manifold outside a rotationally symmetric sphere with
nonnegative constant mean curvature.
\end{thm}

\begin{proof}
This is a direct consequence of Corollary \ref{cor-sec-more-improved-B-M} and Theorem \ref{thm-sec-more-ineq-1}.
\end{proof}

We give a few remarks.

\begin{rem} \label{rem-one-more-proof-pmt}
Theorem \ref{thm-sec-more-ineq-3} gives another way to see the $3$-dimensional positive mass theorem. 
In the context of Proof \RomanNum{3} in Section \ref{sec-pmt}, Theorem \ref{thm-sec-more-ineq-3} gives 
$$  \frac{ \m} {  \c_{r} }   \ge 1  -  \left( \frac{1}{16\pi}  \int_{\Sigma_r}   H^2 \right)^\frac12 = o (1), \ \text{as} \ r \to 0 , $$
where $c_{r}$ is the capacity of a small geodesic ball of radius $r$. 
Hence, $ \m \ge 0 $.
\end{rem}

\begin{rem}
Theorem \ref{thm-sec-more-ineq-3} improves the result of Bray and the author in \cite{BrayMiao08}. 
Under an additional assumption of $ \int_{\Sigma} H^2 \le 16 \pi $, 
in \cite{BrayMiao08} the capacity estimate \eqref{eq-Bray-Miao-06} 
was converted into a Hawking mass estimate
\bee
\m_{_H} (\Sigma) \ge \left[ 1  -  \left( \frac{1}{16\pi}  \int_\Sigma   H^2 \right)^\frac12 \right]  \c_{_\Sigma} 
\eee
and the relation 
$\m \ge \m_{_H} (\Sigma)$
was applied (if $\Sigma$ is outer-minimizing) to obtain \eqref{eq-conformal-ineq-m-b-main-3}.

In the current derivation of \eqref{eq-conformal-ineq-m-b-main-3}, 
we bound the ratio  
$  \m \c_{_\Sigma}^{-1} $ 
via $ \int_{\Sigma} | \nabla u |^2 $ 
and bound  $\int_{\Sigma} | \nabla u |^2 $ via $\int_\Sigma H^2$, 
hence bypassing 
the use of $\m_{_H} (\Sigma)$ in relating $\m$ and $ \c_{_\Sigma}$.
\end{rem}

\begin{rem}
One may re-write \eqref{eq-conformal-ineq-m-b-main-3} as  
\bee
\begin{split}
 \mathfrak{M} (g) : = \frac{ \m} {  \c_{_\Sigma} }  +  \left( \frac{1}{16\pi}  \int_\Sigma   H^2 \right)^\frac12  - 1 \ge 0   .
\end{split}
\eee
This gives  a nonnegative quantity $\mathfrak{M} (g)$ 
on asymptotically flat $3$-manifolds $(M, g)$ with boundary
(under the curvature and topological assumptions).
$\mathcal{M} (g)$ vanishes precisely if $(M,g)$ is rotationally symmetric
with mean-convex boundary. 
\end{rem}

\section{Manifolds with the mass-to-capacity ratio $ \le 1$} \label{sec-m-2-c-less-1}

In this section, prompted by Theorem \ref{thm-sec-more-ineq-3}, we consider a class of manifolds satisfying 
a mass-capacity relation
\be \label{eq-m-c-ratio-relation}
  \m \c_{_\Sigma}^{-1} \le 1 . 
\ee
As we will see later in Proposition \ref{prop-static}, such a class of manifolds includes static metric extensions 
in the context of the Bartnik quasi-local mass \cite{Bartnik89}.

\begin{thm} \label{thm-sec-m-c-l-1}
Let $(M, g)$ be a complete, orientable, asymptotically flat $3$-manifold with boundary $\Sigma$, satisfying a mass-capacity relation 
\bee \label{eq-m-c-ratio-relation-thm}
  \m \c_{_\Sigma}^{-1} \le 1 . 
\eee
Let $ u$ be the harmonic function with $ u = 0 $ at $ \Sigma$ and $ u \to 1 $ near $\infty$.
If $ \Sigma$ is connected,  $H_2 (M, \Sigma) = 0$, and $g$ has nonnegative scalar curvature, then 
\be \label{eq-m-quadratic}
\frac{1}{ 4 \pi}   \int_{\Sigma } H | \nabla u | 
 \ge  \left( 2 - \m  \c_{_{\Sigma}}^{-1}  \right) \left( 1 - \m  \c_{_{\Sigma}}^{-1}   \right) . 
\ee 
Moreover, equality holds if and only if 
$(M, g)$ is isometric to a spatial Schwarzschild manifold outside a rotationally symmetric sphere with
nonnegative constant mean curvature.
\end{thm}

\begin{proof}
For a regular value $t \in [0,1)$, if $ u^{(t)}$ denotes the harmonic function outside $\Sigma_t$ with 
$u^{(t)} = 0 $ at $ \Sigma_t$ and $ u^{(t)} \to 1 $ at $\infty$, then
\be
u^{(t)} = \frac{ u - t}{1 - t}.
\ee
As a result, the capacity $\c_{_{\Sigma_t} } $ of $ \Sigma_t$ is related to that of $ \Sigma $ by
\be \label{eq-capacity-St-and-S}
\c_{_{\Sigma_t}} = \frac{ \c_{_\Sigma} }{ 1 - t }.
\ee
Therefore, by Theorem \ref{thm-sec-more-ineq-3}, 
\be \label{eq-Willmore-capacity-1}
\begin{split}
\left( \frac{1}{16\pi}  \int_{\Sigma_t}   H^2 \right)^\frac12  \ge & \  1 -  \m  \c_{_{\Sigma_t}}^{-1} \\
= & \  1 -  \m  \c_{_{\Sigma}}^{-1} (1 - t )  . 
\end{split} 
\ee
Under the assumption $ \m \c_{_\Sigma}^{-1} \le 1 $, we have 
$$ 1 -  \m  \c_{_{\Sigma}}^{-1} (1 - t ) \ge 0 . $$ 
Hence, \eqref{eq-Willmore-capacity-1} is equivalent to 
\be \label{eq-Willmore-capacity-2}
\begin{split}
 \frac{1}{16\pi}  \int_{\Sigma_t}   H^2   \ge \left[  1 -  \m  \c_{_{\Sigma}}^{-1} (1 - t )  \right]^2.
\end{split} 
\ee

To proceed, we return to the basic identity \eqref{eq-bdry-3} in Section \ref{sec-ineq}.
Given any regular values $ t_1 < t_2 < 1$, by \eqref{eq-bdry-3}, 
\be \label{eq-ratio-bdry-3}
\begin{split}
& \  4 \pi ( t_2 - t_1) + \int_{\Sigma_{t_1} } H | \nabla u |  -  \int_{\Sigma_{t_2} } H | \nabla u |   \\
\ge  & \   \int_{t_1}^{t_2}   \int_{\Sigma_t}  \frac{1}{2}  | \mathring{\Pi} |^2 +  | \nabla u |^{-2}  | \nabla_{_{\Sigma_t} } | \nabla u | |^2 + \frac34 H^2  + \frac12 R  \\
\ge & \   \int_{t_1}^{t_2}   \frac34 \int_{\Sigma_t} H^2 .
\end{split}
\ee
Thus, it follows from \eqref{eq-Willmore-capacity-2} and \eqref{eq-ratio-bdry-3} that
\bee
\label{eq-ratio-bdry-4}
\begin{split}
& \  4 \pi ( t_2 - t_1) + \int_{\Sigma_{t_1} } H | \nabla u |  -  \int_{\Sigma_{t_2} } H | \nabla u |   \\
\ge & \ 12 \pi  \int_{t_1}^{t_2}   \left[  1 -  \m  \c_{_{\Sigma}}^{-1} (1 - t )  \right]^2 .
\end{split}
\eee
Letting $ t_2 \to 1$, by Lemma \ref{lem-limit-1}, we obtain 
\bee \label{eq-ratio-bdry-5}
\begin{split}
 & \ 4 \pi ( 1 - t_1) + \int_{\Sigma_{t_1} } H | \nabla u |  \\
 \ge & \ 12 \pi  \int_{t_1}^{1}   \left[  1 -  \m  \c_{_{\Sigma}}^{-1} (1 - t )  \right]^2 \\
 = & \  12 \pi (1 - t_1) -  12 \pi \m \c_{_\Sigma}^{-1}  ( 1 - t_1)^2  
 + 4 \pi \left( \m  \c_{_{\Sigma}}^{-1} \right)^2 ( 1 - t_1)^3 ,
\end{split}
\eee
or, equivalently 
\be \label{eq-ratio-bdry-6}
\begin{split}
& \  12 \pi  \m \c_{_\Sigma}^{-1}  ( 1 - t_1)  - 4 \pi  \left( \m  \c_{_{\Sigma}}^{-1} \right)^2 ( 1 - t_1)^2  \\
 \ge & \ 8 \pi  -  \frac{1}{1 - t_1}  \int_{\Sigma_{t_1} } H | \nabla u |     .
\end{split}
\ee
In particular, at $ t_1 = 0 $, we have
\bee
\begin{split}
12 \pi  \m \c_{_\Sigma}^{-1}    - 4 \pi  \left( \m  \c_{_{\Sigma}}^{-1} \right)^2  
 \ge  8 \pi  -   \int_{\Sigma } H | \nabla u |     ,
\end{split}
\eee
which proves \eqref{eq-m-quadratic}.

If equality in \eqref{eq-m-quadratic} holds, then equality in \eqref{eq-ratio-bdry-6} holds with $ t_1 = 0 $.
This necessarily implies equality in \eqref{eq-Willmore-capacity-2} holds for a.e. $ t \in [0,1]$.
As a result, at $ t = 0$, 
$$
 \frac{1}{16\pi}  \int_{\Sigma}   H^2   = \left[  1 -  \m  \c_{_{\Sigma}}^{-1}   \right]^2 .
$$
Since $ 1 - \m \c_{_\Sigma}^{-1} \ge 0 $, we conclude by Theorem \ref{thm-sec-more-ineq-3} that 
$(M, g)$ is isometric to a spatial Schwarzschild manifold outside a rotationally symmetric sphere with
nonnegative mean curvature.

Suppose $ (M, g) = \left( \R^3 \setminus \{ | x | < r \} , \left( 1 + \frac{\m}{ 2 | x |} \right)^4 g_{_E} \right) $ 
with mean-convex boundary $\{ |x| = r \}$, then 
\eqref{eq-Willmore-capacity-1} -- \eqref{eq-ratio-bdry-6} all become equality. 
Hence, equality in \eqref{eq-m-quadratic} holds. This completes the proof.
\end{proof}

\begin{rem}
We compare \eqref{eq-m-quadratic} and \eqref{eq-bdry-ineq-m-a}.
If $(M, g)$ has $\m = 0 $,  \eqref{eq-m-quadratic} is as the same as \eqref{eq-bdry-ineq-m-a},
both of which reduces to 
$$\int_{\Sigma} H | \nabla u | \ge 8 \pi  . $$
For $(M, g)$ with $\m \ne 0 $, \eqref{eq-m-quadratic} improves  \eqref{eq-bdry-ineq-m-a}
by unveiling the additional term
$$ 4 \pi ( \m \c_{_\Sigma}^{-1} )^2 . $$
\end{rem}

\begin{rem} \label{rem-good-condition}
Condition \eqref{eq-m-c-ratio-relation} is a global condition on the triple $(M, g, \Sigma)$.
It has a feature of being inheritable to other surfaces enclosing
$\Sigma$. More precisely, 
$$  \m \c_{_\Sigma}^{-1} \le 1 \ \Longrightarrow \ \m \c_{_S}^{-1} \le 1 $$ 
for any other surfaces $S$ in $M$ enclosing $ \Sigma$.
This follows from the fact $ \c_{_\Sigma} \le \c_{_S} $, which is a consequence of the variational
characterization of the surface capacity.
\end{rem}

\begin{rem}
Theorem \ref{thm-sec-m-c-l-1} shows a necessary condition of $ \m \c_{_\Sigma} \le 1 $ 
is $ \int_{\Sigma} H | \nabla u | \ge 0 $.
Therefore, by Remark \ref{rem-good-condition},
\be
\m \c_{_\Sigma}^{-1} \le 1 \ \Longrightarrow \ \int_{S} H | \nabla u_{_S} | \ge 0 ,
\ee
for any surfaces $S$  enclosing $\Sigma$. Here $ u_{_S}$ denotes the harmonic
function in the exterior of $S$, with $ u_{_S} = 0 $ at $S$ and $ u_{_S} \to 1 $ at $\infty$.
\end{rem}

Manifolds $(M, g)$ satisfying $ \m \c_{_\Sigma}^{-1} \le 1 $ include regions, in a spatial Schwarzschild manifold with positive mass, 
which are the exterior to a surface enclosing the horizon. That is,  if 
$$ (M_\m, g_\m ) =   \left( \R^3 \setminus \left\{ | x | < \frac12 \m  \right\} , \left( 1 + \frac{\m}{ 2 | x |} \right)^4 g_{_E} \right) $$
with $ \m > 0 $ and if $\Sigma \subset  M_\m $ 
is a closed surface bounding some region $D$ with the horizon 
$\Sigma_{h} = \{ | x | = \frac12 \m \}, $
then $(M_\m \setminus D, g_\m )$ satisfies
\bee
\m \c_{_\Sigma}^{-1} \le 1 . 
\eee
This is because of $ \ c_{_{\Sigma_h}} = \m  $ on $(M_\m, g_\m)$ and 
$ \c_{_{\Sigma_h} } \le \c_{_\Sigma} $. 

To put the next corollary of Theorem \ref{thm-sec-m-c-l-1} in context, 
we mention a few additional facts on $(M_\m, g_\m)$ . 
Let $ \Sigma_r = \{ | x | = r \} \subset M_\m$.
The mean curvature $H_r$ of $ \Sigma_r $ equals 
$$ H_r = k^{-2} \left( 2 k^{-1}  - 1 \right) 2 r^{-1} , $$
where $ k > 0 $ is the constant determined by 
$ \displaystyle  \m = 2 r ( k - 1 )  $.
The product $ \m H_r$ satisfies
\bee
\m H_r  = 2 k^{-1} \left( 2 k^{-1}  - 1 \right)  2 (1-k^{-1}) .
\eee
The capacity $\c_r$ of $ \Sigma_r$ is given by 
\bee
\c_r = r + \frac{\m}{2},  \ \ \text{and hence,} \ \ 
\m \c_r^{-1} = 2 ( 1 - k^{-1} ) < 2 .
\eee
Thus, $ \m H_r $ and $ \m \c_{r}^{-1}$ are related by
\be \label{eq-mH-expression-S}
\m H_r = \left( 2 - \m \c_r^{-1} \right) \left( 1 - \m \c_r^{-1} \right)  \m \c_r^{-1} . 
\ee
As a function of $r$, calculation shows
\be \label{eq-max-mH-S}
\max_{ \frac12 \m \le r < \infty} \m H_r = \frac{2}{ 3 \sqrt{ 3} } ,
\ee
where this maximum is achieved uniquely at
\be
r_p =  \left( 1 + \frac{ \sqrt{3} }{2} \right)  \m , \ \text{satisfying}  \ \left( 1 + \frac{ \m }{2 r_p} \right)^2 r_p = 3 \m . 
\ee
The sphere $ \{ |x| = r_p \}$ is often known as the photon sphere in $(M_\m, g_\m)$.
The mass-to-capacity ratio at  $\Sigma_{r_p}$ is given by
\be \label{eq-q-photon}
\m \c_{r_p}^{-1} = 1 - \frac{1}{  \sqrt{3} } . 
\ee

The following corollary gives a partial classification or comparison result for manifolds with $\m \c_{_\Sigma} \le 1 $, 
depending on the maximum of the boundary mean curvature.

\begin{cor} \label{cor-sec-m-c-l-1}
Let $(M, g)$ be a complete, orientable, asymptotically flat $3$-manifold with boundary $\Sigma$, 
with the mass-to-capacity ratio satisfying 
$$  0 <  \m \c_{_\Sigma}^{-1} \le 1 . $$
Suppose $ \Sigma$ is connected,  $H_2 (M, \Sigma) = 0$, and $g$ has nonnegative scalar curvature.
Then 
\begin{enumerate}

\vspace{.1cm}

\item[(i)] either $(M, g)$ is isometric to a spatial Schwarzschild manifold outside the horizon; 

\vspace{.1cm}

\item[(ii)]  or $ \displaystyle H_{max} =  \max_{\Sigma} H  > 0 $ and one of the following holds:

\begin{enumerate}

\vspace{.1cm}

\item $  \m H_{max} < \frac{2}{3 \sqrt{3} }  $ and 
$$ \displaystyle \c_{_\Sigma} \le \c_{r_1} \  \text{or} \ \ \c_{_\Sigma} \ge \c_{r_2}  . $$ 
Here $ \c_{r_i}$ is the capacity of the sphere $\Sigma_{r_i} = \{ |x| = r_i \}$, $ i = 1, 2$,  
in the spatial Schwarzschild manifold 
$$ (M_\m, g_\m ) =   \left( \R^3 \setminus \left\{ | x | < \frac12 \m  \right\} , \left( 1 + \frac{\m}{ 2 | x |} \right)^4 g_{_E} \right) $$
which has the same mass as $(M,g)$,  and the constants $r_1, r_2$ are chosen so that 
\be \label{eq-def-r-1-and-r-2}
\ \ \ \ \ \ H_{r_1} = H_{r_2} = H_{max} \ \ \text{and} \ \ \  \frac12 \m < r_1 < ( 1 + \frac{1 }{2}  \sqrt{3} ) \, \m < r_2   , 
\ee
where $H_{r_i}$ is the  mean curvature of $\Sigma_{r_i}$ in $(M_\m, g_\m)$. 
Moreover, $ \c_{_\Sigma} = \c_{r_i}$ for an  $r_i$ if and only if $ (M, g)$ is isometric to 
$(M_\m, g_\m)$ outside $\Sigma_{r_i}$;

\vspace{.2cm}

\item $ \displaystyle  \m H_{max} \ge \frac{2}{3 \sqrt{3} } $ and equality holds if and only 
if $(M, g)$ is isometric to the spatial Schwarzschild manifold $(M_\m , g_\m)$ outside the photon sphere
$ \left\{ |x| = ( 1 + \frac{1 }{2}  \sqrt{3} ) \, \m \right\} $.

\end{enumerate}

\end{enumerate}

\end{cor}

\begin{proof}
Let $ q = \m \c_{_\Sigma}^{-1} \in (0, 1]$.
By Theorem \ref{thm-sec-m-c-l-1}, 
\be \label{eq-mass-H}
\begin{split}
\frac{1}{ 4 \pi}   \int_{\Sigma } H | \nabla u | 
 \ge & \ \left( 2 - q \right) \left( 1 - q  \right) . 
\end{split}
\ee
In particular, $  \int_\Sigma H | \nabla u | \ge 0 $.  
As $ | \nabla u | > 0 $ along $ \Sigma$,  we  have  $ H_{max} \ge 0  $
and $ H_{max} = 0 $ if and only if $ H = 0 $. 
In the latter case, Theorem \ref{thm-sec-more-ineq-3} shows $ q \ge 1 $. 
Therefore, $ q = 1 $, and by Theorem \ref{thm-sec-more-ineq-3}, $(M, g)$ is isometric to a spatial Schwarzschild manifold with positive mass outside the horizon.

In what follows, we suppose $ H_{max} > 0 $. 
Since $ \int_{\Sigma} | \nabla u | = 4 \pi \c_{_\Sigma}$, 
\eqref{eq-mass-H} implies 
\be
H_{max} \,  \c_{_\Sigma} \ge \left( 2 - q  \right) \left( 1 - q  \right) .
\ee
As $ \m > 0 $, this gives
\be \label{eq-mass-H-m-p}
\m \, H_{max} \ge \left( 2 - q \right) \left( 1 - q \right) q .
\ee
As a result, either  
\be \label{eq-big-mH}
\m \, H_{max} \ge  \frac{2}{3 \sqrt{3} } = \max_{ x \in [0, 1] } ( 2 - x ) ( 1 - x ) x  , 
\ee
or
\be \label{eq-middle-mH}
 0 < \m \, H_{max} <   \frac{2}{3 \sqrt{3} } . 
\ee
If  \eqref{eq-big-mH} holds with equality, then 
\bee
H_{max} \c_{\Sigma} =  \frac{1}{ 4 \pi}   \int_{\Sigma } H | \nabla u | 
=  \left( 2 - q \right) \left( 1 - q  \right) ,
\eee
with 
$ q =  1 - \frac{1}{\sqrt{3} }  $.
By Theorem \ref{thm-sec-m-c-l-1} and the fact \eqref{eq-max-mH-S} -- \eqref{eq-q-photon}, $(M, g)$ is isometric to 
a spatial Schwarzschild manifold with the photon sphere boundary.

Next, we suppose \eqref{eq-middle-mH} holds. 
Let $ r_i$, $i=1,2$, be the constants given in \eqref{eq-def-r-1-and-r-2}.
It follows from \eqref{eq-mH-expression-S} and \eqref{eq-mass-H-m-p} that
\bee
\begin{split}
& \ \left( 2 - \m \c_{r_i}^{-1} \right) \left( 1 - \m \c_{r_i}^{-1} \right)  \m \c_{r_i}^{-1} \\
= & \ \m H_{r_i} = \m H_{max} \\
\ge & \ ( 2 - q) (1 - q) q .
\end{split} 
\eee
Analyzing the function $f(x) = (2 -x)(1-x)x $ and using the assumption $ 0 < q \le 1 $, we conclude
\bee
q \le \m \c_{r_2}^{-1} \ \  \text{or} \  \ q \ge \m {\c}_{r_1}^{-1}, 
\eee
or equivalently 
\be
\c_{r_2} \le \c_{_\Sigma} \ \ \text{or} \  \ \c_{r_1} \ge \c_{_\Sigma} . 
\ee
If $ \c_{_\Sigma} = \c_{r_i}$ for an $r_i$, then 
\bee
H_{max} \c_{_\Sigma} =  \frac{1}{ 4 \pi}   \int_{\Sigma } H | \nabla u | 
=  \left( 2 - q \right) \left( 1 - q  \right) ,
\eee
with 
$ q =  \m \c_{r_i}^{-1}  $.
By Theorem \ref{thm-sec-m-c-l-1}, $(M, g)$ is isometric to 
a spatial Schwarzschild manifold with boundary $\{ |x| = r_i \}$.
This completes the proof. 
\end{proof}

\begin{rem}
Corollary \ref{cor-sec-m-c-l-1} can be applied to manifolds $(M, g)$ with CMC boundary, i.e. $\Sigma$ 
has constant mean curvature.
In this case, it might be interesting to identify  
$ \sup_{(M, g)} \m H $.
\end{rem}

Next, we mention some other implications of \eqref{eq-m-c-ratio-relation}
which are  corollaries of Theorems \ref{thm-BM07} and  \ref{thm-sec-more-ineq-1}.

\begin{cor} \label{cor-sec-m-c-l-2}
Let $(M, g)$ be a complete, orientable, asymptotically flat $3$-manifold with boundary $\Sigma$, satisfying 
$$   \m \c_{_\Sigma}^{-1} \le 1 . $$  
Suppose $ \Sigma$ is connected,  $H_2 (M, \Sigma) = 0$, and $g$ has nonnegative scalar curvature. Then
\begin{enumerate}
\item[(i)] $ \displaystyle 
\m \le \frac{r_{_\Sigma}} {2} \left[ 1 + \left( \frac{1}{16\pi} \int_\Sigma H^2 \right)^\frac12 \right] ,
$
where $ r_{_\Sigma}$ is the area-radius of $ \Sigma $; and 

\item[(ii)]  
$ \displaystyle  \frac{1}{\pi} \int_{\Sigma} | \nabla u |^2 \ge 1 $,
where  $u$ is the harmonic function on $(M, g)$ with $ u = 0 $ at $ \Sigma$ and $ u \to 1 $ near $\infty$.

\end{enumerate} 
Moreover, equality holds in either inequality  if and only if 
$(M, g)$ is isometric to a spatial Schwarzschild manifold outside the horizon.
\end{cor}

\begin{proof} 
Inequalities in (i),  (ii) follow from  \eqref{eq-Bray-Miao-06} in Theorem \ref{thm-BM07}, 
 \eqref{eq-conformal-ineq-m-b-main} in  Theorem \ref{thm-sec-more-ineq-1}, respectively. 
The rigidity part follows from the rigidity conclusion in
Theorem \ref{thm-BM07} or  Theorem \ref{thm-sec-more-ineq-1}, together with 
the extra information $\m = \c_{_\Sigma}$.
\end{proof}

\begin{rem} \label{rem-long-cylinder}
Heuristically, (ii) of Corollary \ref{cor-sec-m-c-l-2} suggests 
the condition 
$$ \m \c_{_\Sigma}^{-1} \le 1 $$ 
may rule out manifolds having  long cylindrical neighborhoods shielding  the boundary. 
The following is a simple example. Suppose $ \Sigma$ is a sphere or a torus and $\gamma$ is a metric of nonnegative Gauss curvature on $\Sigma$. Given a constant $ L > 0$, consider the product manifold 
$$ (P , g_{_P} ) = ( \Sigma \times [0, L], \gamma + d t^2)  .$$ 
If  $(M, g)$ contains a neighborhood $U$ of $\p M $ so that 
$(U, g)$ is isometric to $(P, g_{_P})$ with $ \p M = \Sigma \times \{ 0 \}$, then 
 one can consider the harmonic function  $v  $ on $(P, g_{_P})$ with 
$v = 0 $ on $  \Sigma \times \{ 0 \}$ and $ v = 1 $ on $\Sigma \times \{ L \}$.
By the maximum principle,  $ | \nabla u | \le | \nabla v | $. Hence
\be
 4 r_{_\Sigma}^2 L^{-2}  =  \frac{1}{\pi } \int_{\Sigma} | \nabla v |^2
\ge \frac{1}{\pi } \int_{\Sigma} | \nabla u |^2 ,
\ee
where $ 4 \pi r_{_\Sigma}^2 $ is the area of $(\Sigma, \gamma)$. Therefore, 
if $ (M, g)$ satisfies condition \eqref{eq-m-c-ratio-relation}, then 
(ii) of Corollary \ref{cor-sec-m-c-l-2} shows $ L \le 2 r_{_\Sigma}$. 
\end{rem}

One can always construct manifolds satisfying  \eqref{eq-m-c-ratio-relation}
by cutting off a compact set in a given asymptotically flat $(M, g)$.
For instance, if a  triple $(M, g, \Sigma)$ has $ \m > \c_{_\Sigma} $,   
 then, by \eqref{eq-capacity-St-and-S}, 
 the exterior of $\Sigma_t $  in $(M, g)$  satisfies $ \m \le \c_{_{\Sigma_t}}$
for any regular $ t \ge 1 -  \c_{_\Sigma}  \m^{-1} $.

The complement of a finite domain enclosing the Schwarzschild horizon in $(M_\m , g_{\m})$ is an example of a static extension
in the context of Bartnik's quasi-local mass \cite{Bartnik89}. 
Here an asymptotically flat $(M, g)$ is called static (see \cite{Corvino} for instance)
if there is a nontrivial function $N  $ on $(M,g)$, referred as a static potential,  such that
\be \label{eq-static-system}
\left\{
\begin{array}{rcl}
N \Ric & = & \nabla^2 N \\
\Delta N & = & 0 , \\
\end{array}
\right.
\ee
where $\Ric$ denotes the Ricci curvature of $g$.  These spaces necessarily have zero scalar curvature. 

The next proposition, among other things,  shows an asymptotically flat manifold with boundary, 
admitting a positive static potential, satisfies  \eqref{eq-m-c-ratio-relation}.

\begin{prop} \label{prop-static}
Let $(M, g)$ be a complete, orientable, asymptotically flat $3$-manifold with boundary $\Sigma$.
Suppose there is a static potential $N $ that is positive in  the interior of $(M, g)$. Then
$$  \m \c_{_\Sigma}^{-1} \le 1 . $$
If $\Sigma$ is connected and $H_2 (M, \Sigma) = 0 $, then 

\begin{enumerate}

\item[(i)] $ \m \le \frac{r_{_\Sigma}} {2} \left[ 1 + \left( \frac{1}{16\pi} \int_\Sigma H^2 \right)^\frac12 \right] ,
$ where $r_{_\Sigma}$ is the area-radius of $\Sigma$; and 

\item[(ii)] any closed, regular level set $S_t = N^{-1} (t)$ is connected and enclosing $\Sigma$; 
if $N$ is normalized so that $ N \to 1 $ at $\infty$, then,  along $S_t$, 
$$   N^2  \le \frac{1}{16 \pi} \int_{S_t} H^2 ,  $$
$$ 1-N^2 \le  \left( \frac{1}{\pi} \int_{S_t} | \nabla N |^2  \right)^\frac12 \le ( 1 -  N ) \left[ 1 + \left( \frac{1}{16 \pi} \int_{S_t} H^2 \right)^\frac12 \right] ,  $$
 and 
$$
 N ( 1 - N) ( 1 + N) \le \frac{1}{4\pi} \int_{S_t} H | \nabla N | .
 $$
Moreover, equality holds in any of these inequalities if and only if the exterior of $S_t$ in $(M, g)$ 
 is isometric to a spatial Schwarzschild manifold outside a rotationally symmetric sphere with
nonnegative constant mean curvature.
\end{enumerate} 
\end{prop}

\begin{proof}
The static system \eqref{eq-static-system} and the assumption $N > 0 $ near $\infty$ imply that, 
upon multiplying $N$ by a constant, 
\bee
N = 1 - \frac{\m}{|x|} + o (|x|^{-1} ) , 
\eee
where $ \m $ is the mass of $(M, g)$ (see \cite{BM87, MT14} for instance).
If in addition $N \ge 0 $ at $ \Sigma $, then $ \m \le \c_{_\Sigma}  $ 
by the maximum principle.

The mass estimate in (i)  follows from (i) of Corollary \ref{cor-sec-m-c-l-2}. 

Suppose $S_t$ is a closed, regular level set of $N$. By the topological assumption on $M$
and the fact $N$ is harmonic, 
$S_t$ only has one connected component and it encloses $\Sigma$. 
Let $ E_t $ denote the exterior of $S_t$ in $M$.
The inequalities in (ii), with the rigidity conclusions, follow from applying Theorem \ref{thm-sec-more-ineq-3}, Corollary \ref{cor-sec-more-improved-B-M}, 
Theorem \ref{thm-sec-more-ineq-1}, and Theorem \ref{thm-sec-m-c-l-1}, respectively, 
 to 
$$ u^{(t)} = \frac{ N - t}{1-t}$$
on $(E_t, g)$ and using the fact $ \displaystyle \c_{_{S_t}} = \frac{ \m }{1-t} $. 
\end{proof}

In the context of the Bartnik mass \cite{Bartnik89}, asymptotically flat extensions are often assumed 
to have no closed minimal surface enclosing the boundary 
to prevent the infimum of the mass over all extensions from being trivially zero. 
We note here, if an asymptotically flat $3$-manifold $(M, g)$ with boundary $\Sigma$ satisfies the mass-to-capacity 
relation
$
 \m \c_{_\Sigma}^{-1} \le 1 , 
$
 then necessarily there are no closed minimal surfaces enclosing $\Sigma$.
This is because, if such a minimal surface $S$ exists, then $ \m \ge \c_{_S}$ by the result of Bray \cite{Bray02}.
On the other hand, $ \c_{_S} > \c_{_\Sigma}$. Hence, $ \m > \c_{_\Sigma}$, violating \eqref{eq-m-c-ratio-relation}.

Considering this and Proposition \ref{prop-static},  we think  manifolds satisfying condition \eqref{eq-m-c-ratio-relation}
are worthy of further study. The mass-to-capacity ratio condition $ \m \c_{_\Sigma}^{-1} \le 1 $
may serve as an alternative to the no-minimal-surface or outer-minimizing conditions in the formulation of the Bartnik mass.

\vspace{.1cm}

\appendix

\section{Regularization and integration} \label{sec-regularization}

In this appendix, we give the regularization arguments
that can be used to verify the monotonicity of $\Psi (t)$, $\mathcal{A}(t)$ and $\mathcal{B}(t)$ 
in Section \ref{sec-ineq}. 

\begin{lma} \label{lem-reg}
Let $ u $ be a harmonic function on a compact Riemannian manifold $(\Omega, g)$ with boundary $\p \Omega$. 
Suppose $ \max_{\Omega} u < 1 $. Then
\be \label{eq-reg-for-psi}
\begin{split}
 \int_{\p \Omega} \frac{  | \nabla u |   }{ 1 - u }  \frac{\p u}{\p \zeta} 
=  \int_\Omega \frac{ | \nabla u |^3  }{ ( 1 - u )^2 }    + 
\int_{ \{ \nabla u \ne 0 \}  \subset \Omega } \frac{  \nabla^2 u ( \nabla u , \nabla u  )  } { (1 - u)  | \nabla u |  }   
\end{split}
\ee
and
\be \label{eq-reg-for-B}
\begin{split}
 \int_{\p \Omega} \frac{  | \nabla u |   }{ (1 - u)^3 }  \frac{\p u}{\p \zeta} 
=  \int_\Omega \frac{ 3  | \nabla u |^3  }{ ( 1 - u )^4 }    + 
\int_{ \{ \nabla u \ne 0 \}  \subset \Omega } \frac{  \nabla^2 u ( \nabla u , \nabla u  ) } { (1 - u)^3 | \nabla u |  }     . 
\end{split}
\ee
Here $ \zeta $ denotes the unit normal to $ \p \Omega$ pointing out of $\Omega$. 
\end{lma}

\begin{proof}
Given any constant  $\epsilon > 0$, one has
\bee 
\div \left( \frac{ \sqrt{ | \nabla u |^2 + \epsilon}   }{ 1 - u }  \nabla u \right)  
=  \frac{ \sqrt{ | \nabla u |^2 + \epsilon} }{ ( 1 - u )^2 } | \nabla u |^2 + 
\frac{1} { 1 - u} \frac{ \nabla^2 u ( \nabla u, \nabla u )  }{\sqrt{ | \nabla u |^2 + \epsilon}}    .
\eee
Therefore,
\be \label{eq-divergence-thm-1}
\int_{\p \Omega} \frac{ \sqrt{ | \nabla u |^2 + \epsilon}   }{ 1 - u }  \frac{\p u}{\p \zeta}  \\
=  \int_\Omega \frac{ \sqrt{ | \nabla u |^2 + \epsilon} }{ ( 1 - u )^2 } | \nabla u |^2 + 
\int_{  \Omega  } \frac{1} { 1 - u} \frac{  \nabla^2 u ( \nabla u, \nabla u ) }{\sqrt{ | \nabla u |^2 + \epsilon}}   . 
\ee
For the third term in \eqref{eq-divergence-thm-1}, one notes
\bee
\begin{split}
 \int_{  \Omega  } \frac{1} { 1 - u} \frac{  \nabla^2 u ( \nabla u, \nabla u ) }{\sqrt{ | \nabla u |^2 + \epsilon}} 
= & \ \int_{ \{ \nabla u \ne 0 \}  \subset  \Omega  } \frac{1} { 1 - u} \frac{  \nabla^2 u ( \nabla u, \nabla u ) }{\sqrt{ | \nabla u |^2 + \epsilon}} \\
= & \ \int_{ \{ \nabla u \ne 0 \}  \subset  \Omega  }   \frac{1} { 1 - u}  \, \nabla^2 u \left( \frac{\nabla u}{ | \nabla u | }, \frac{ \nabla u }{ | \nabla u | } \right) 
 \frac{   | \nabla u |^2 }{\sqrt{ | \nabla u |^2 + \epsilon}} .
\end{split}
\eee
Thus, taking $ \epsilon \to 0 $ in \eqref{eq-divergence-thm-1} proves \eqref{eq-reg-for-psi}. 

Similarly, to show \eqref{eq-reg-for-B},  one has
\bee 
\div \left[ \frac{ \sqrt{ | \nabla u |^2 + \epsilon}   }{ (1 - u)^3 }  \nabla u \right] 
=  \frac{ 3  \sqrt{ | \nabla u |^2 + \epsilon} }{ ( 1 - u )^4 } | \nabla u |^2 + 
\frac{1} { (1 - u)^3 } \frac{ \nabla^2 u ( \nabla u, \nabla u )  }{\sqrt{ | \nabla u |^2 + \epsilon}}    .
\eee
Consequently,
\be \label{eq-divergence-thm-2}
\int_{\p \Omega} \frac{ \sqrt{ | \nabla u |^2 + \epsilon}   }{ (1 - u)^3 }  \frac{\p u}{\p \zeta}  \\
=  \int_\Omega \frac{ 3  \sqrt{ | \nabla u |^2 + \epsilon} }{ ( 1 - u )^4 } | \nabla u |^2 + 
\int_{  \Omega  } \frac{1} { (1 - u)^3 } \frac{ \nabla^2 u ( \nabla u, \nabla u )  }{\sqrt{ | \nabla u |^2 + \epsilon}} . 
\ee
The third term above satisfies 
\bee
\begin{split}
\int_{  \Omega  } \frac{1} { (1 - u)^3 } \frac{ \nabla^2 u ( \nabla u, \nabla u )  }{\sqrt{ | \nabla u |^2 + \epsilon}}
= & \ \int_{ \{ \nabla u \ne 0 \}  \subset  \Omega  }   \frac{1} { (1 - u)^3 }  \, \nabla^2 u \left( \frac{\nabla u}{ | \nabla u | }, \frac{ \nabla u }{ | \nabla u | } \right) 
 \frac{   | \nabla u |^2 }{\sqrt{ | \nabla u |^2 + \epsilon}} .
\end{split}
\eee
Thus, \eqref{eq-reg-for-B} follows by taking taking $ \epsilon \to 0 $ in \eqref{eq-divergence-thm-2}.
\end{proof}

\begin{lma} \label{lem-reg-3}
Let $ u $ be a harmonic function on a compact, orientable, Riemannian $3$-manifold $(\Omega, g)$ with boundary $\p \Omega$. 
Suppose $\max_\Omega u < 1 $ and $u$ equals a constant on each connected component of $\p \Omega$. 
Then
\be \label{eq-weighted}
\begin{split}
\int_{\p \Omega}   \frac{ H  | \nabla u |  } { ( 1 - u )^2}  \le  \int_{t_1}^{t_2} 
\frac{1}{(1-t)^2}   \left\{  \int_{\Sigma_t}  \left[ \frac{2  H | \nabla u | }{1-t}    
-    \frac12  \left(  \frac{ | \nabla^2 u |^2 }{  | \nabla u |^2 } +  R \right)  \right] +  2 \pi  \chi(\Sigma_t) \right\}   .
\end{split}
\ee
Here the mean curvature $H$ of $\p \Omega$ is taken with respect to the 
unit normal $\zeta$ pointing out of $\Omega$, 
the mean curvature $H$ of a regular level set $\Sigma_t$ is taken with respect to
$ | \nabla u |^{-1} | \nabla u |$,  $\chi (\Sigma_t)$ is the Euler characteristic of $\Sigma_t$, 
 $t_1 = \min_{\Omega} u $, and $ t_2 = \max_{\Omega} u $.
\end{lma}

\begin{proof}
For any constant  $\epsilon > 0$, one has
\bee
\begin{split}
 \div \left[ \frac{ \nabla \sqrt{ | \nabla u |^2 + \epsilon } }{ ( 1 - u )^2 }  \right] 
=  \frac{ \Delta \sqrt{  | \nabla u |^2 + \epsilon }  }{ ( 1 - u )^2} 
+  \frac{2 \,  \nabla^2 u ( \nabla u , \nabla u ) }{ ( 1 - u )^3  \sqrt{ | \nabla u |^2 + \epsilon }  } .
\end{split}
\eee
Therefore,  
\be \label{eq-int-zeta}
\int_{\p \Omega}   \frac{  \p_\zeta  \sqrt{ | \nabla u |^2 + \epsilon }   }{ ( 1 - u )^2 }   \\
= \int_{\Omega} \frac{ \Delta \sqrt{  | \nabla u |^2 + \epsilon }  }{ ( 1 - u )^2} 
+  \int_{\Omega} \frac{2 \,  \nabla^2 u ( \nabla u , \nabla u ) }{ ( 1 - u )^3  \sqrt{ | \nabla u |^2 + \epsilon }  } .
\ee

As $ u $ is constant on each connected component of $\p \Omega$, direct calculations gives
\bee
\p_\zeta \sqrt{ | \nabla u |^2 + \epsilon }  = - \frac{ | \nabla u |^2} {  \sqrt{ | \nabla u |^2 + \epsilon } } H.
\eee
(See Lemma 2.1 in \cite{HMT20} for instance.)
Thus,
\be \label{eq-limit-bdry-HMT}
\lim_{\epsilon \to 0} \int_{\p \Omega}   \frac{  \p_\zeta  \sqrt{ | \nabla u |^2 + \epsilon }   }{ ( 1 - u )^2 } 
= -  \int_{\p \Omega}   \frac{H  | \nabla u | } { ( 1 - u )^2} . 
\ee

As in the proof of the previous lemma, taking $\epsilon \to 0$ in the third term in \eqref{eq-int-zeta} gives 
\be \label{eq-as-in-previous-pf}
\begin{split}
\lim_{ \epsilon \to 0} \int_{\Omega} \frac{2 \,  \nabla^2 u ( \nabla u , \nabla u ) }{ ( 1 - u )^3  \sqrt{ | \nabla u |^2 + \epsilon }  } 
=  & \ \int_{ \{ \nabla u \neq 0 \} \subset \Omega} \frac{2 \,  \nabla^2 u ( \nabla u , \nabla u ) }{ ( 1 - u )^3  | \nabla u | } \\
= & \ - \int_{t_1}^{t_2}  \frac{2}{(1-t)^3 } \int_{\Sigma_t}  H | \nabla u | ,
\end{split}
\ee
where the second equation follows from the coarea formula and \eqref{eq-lap-conse}.

To deal with the second term in \eqref{eq-int-zeta}, we follow an argument of Stern \cite{Stern19}. 
Let $ \mathcal{C} $ denote the set of critical values of $u$ in $[t_1, t_2]$.
Let $ W  $ denote an open set of $ [ t_1, t_2]$
such that $ W$ contains $ \mathcal{C} $. Let $ D $ be the complement of $ W$ in  $[ t_1, t_2 ] $.

On $u^{-1} (D)$, $ \frac{ \Delta \sqrt{  | \nabla u |^2 + \epsilon }  }{ ( 1 - u )^2 |\nabla u | }$ is integrable.
By coarea formula, 
\bee
\int_{u^{-1} ( D ) }  \frac{ \Delta \sqrt{  | \nabla u |^2 + \epsilon }  }{ ( 1 - u )^2 }
= \int_{ D} \int_{\Sigma_t}  \frac{ \Delta \sqrt{  | \nabla u |^2 + \epsilon }  }{ ( 1 - u )^2  | \nabla u | } .
\eee
Along $\Sigma_t$ which a regular level set of $u$, by equation (14) in \cite{Stern19},
\be
\Delta \sqrt{ | \nabla u |^2 + \epsilon} \ge \frac{1}{2 \sqrt{ | \nabla u |^2 + \epsilon} } 
\left[ | \nabla^2 u |^2 + (R - 2 K _{_{\Sigma_t} } ) | \nabla u |^2 \right] ,
\ee
where $ K_{_{\Sigma_t} } $ is the Gauss curvature of $\Sigma_t$. 
Thus, 
\be  \label{eq-int-over-u-D}
\int_{u^{-1} ( D ) }  \frac{ \Delta \sqrt{  | \nabla u |^2 + \epsilon }  }{ ( 1 - u )^2 }
\ge
 \int_{ D} \int_{\Sigma_t}  \frac{ 1}{ ( 1 - u )^2   | \nabla u |   } 
 \frac{\left[ | \nabla^2 u |^2 + (R - 2 K_{_{\Sigma_t} } ) | \nabla u |^2 \right] }{  
 2 \sqrt{ | \nabla u |^2 + \epsilon}    } .
\ee
With $ W$ fixed, letting $ \epsilon \to 0$ in \eqref{eq-int-over-u-D} gives
\be \label{eq-closed-set-D}
\begin{split}
\liminf_{\epsilon \to 0} \int_{u^{-1} ( D ) }  \frac{ \Delta \sqrt{  | \nabla u |^2 + \epsilon }  }{ ( 1 - u )^2 }
\ge & \ 
 \int_{ D} \int_{\Sigma_t}  \frac{\left[ | \nabla^2 u |^2 + (R - 2 K_{_{\Sigma_t} } ) | \nabla u |^2 \right] }{ ( 1 - u )^2  \, 2 | \nabla u |^2   } \\
 = & \  \int_{ D}  \frac{1}{(1-t)^2}  
 \left[ \int_{\Sigma_t} \frac12 \left(   | \nabla u |^{-2}   | \nabla^2 u |^2 + R \right) - 2 \pi \chi (\Sigma_t) \right]  ,
\end{split}
\ee
where one also used the Gauss-Bonnet theorem. 

To estimate the integral on $u^{-1} (W)$, one notes 
\bee
\begin{split}
 \Delta \sqrt{ | \nabla u |^2 + \epsilon } 
= & \ \frac{1}{ \sqrt{ | \nabla u |^2 + \epsilon } } 
\left[ | \nabla^2 u |^2 + \Ric (\nabla u , \nabla u ) - \frac{1} { 4 ( |\nabla u |^2 + \epsilon) } | \nabla | \nabla u |^2 |^2 \right] \\
\ge & \ \frac{1}{ \sqrt{ | \nabla u |^2 + \epsilon } } \Ric (\nabla u , \nabla u )  .
\end{split}
\eee
This implies  
\be \label{eq-open-set-W}
\begin{split}
\int_{u^{-1} ( W ) }  \frac{ \Delta \sqrt{  | \nabla u |^2 + \epsilon }  }{ ( 1 - u )^2 }
\ge & \   -  \max_{\Omega} | \Ric | \int_{u^{-1} ( W ) }  \frac{ | \nabla u |  }{ ( 1 - u )^2 } \\
=  & \ -  \max_{\Omega} | \Ric | \int_{ W  }  \int_{\Sigma_t}  \frac{ 1 }{ ( 1 - u )^2 } .
\end{split}
\ee

It follows from \eqref{eq-closed-set-D} and \eqref{eq-open-set-W} that
\bee
\begin{split}
 \lim_{\epsilon \to 0} \int_{\Omega} \frac{ \Delta \sqrt{  | \nabla u |^2 + \epsilon }  }{ ( 1 - u )^2} \ge & \ 
 \int_{ D}  \frac{1}{(1-t)^2}  
 \left[ \int_{\Sigma_t} \frac12 \left(   | \nabla u |^{-2}   | \nabla^2 u |^2 + R \right) - 2 \pi \chi (\Sigma_t) \right]  \\
& \  -  \max_{\Omega} | \Ric | \int_{ W  }  \int_{\Sigma_t}  \frac{ 1 }{ ( 1 - u )^2 } .
\end{split}
\eee
As $ \int_\Omega \frac{ | \nabla u |} { ( 1 - u)^2} < \infty$, 
by choosing the measure of $ W$  to be arbitrarily small, one has
\be \label{eq-limit-mid-term}
\begin{split}
 \lim_{\epsilon \to 0} \int_{\Omega} \frac{ \Delta \sqrt{  | \nabla u |^2 + \epsilon }  }{ ( 1 - u )^2} 
\ge & \ 
 \int_{t_1}^{t_2}  \frac{1}{(1-t)^2}  
 \left[ \int_{\Sigma_t} \frac12 \left(   | \nabla u |^{-2}   | \nabla^2 u |^2 + R \right) - 2 \pi \chi (\Sigma_t) \right]  .
\end{split}
\ee
The lemma now follows from \eqref{eq-int-zeta},  \eqref{eq-limit-bdry-HMT}, \eqref{eq-as-in-previous-pf} and \eqref{eq-limit-mid-term}.
\end{proof}

\begin{rem}
\eqref{eq-weighted} may be viewed as a weighted version of the identity \eqref{eq-bdry-1}.
\end{rem}

\begin{prop} \label{prop-reg-monotone}
Let $ (\Omega, g)$ be a connected, compact, orientable, Riemannian $3$-manifold with boundary $\p \Omega$.
Suppose $\p \Omega$ is the disjoint union of  two nonempty pieces $S_1 $ and $S_2$.
Let $ u $ be a harmonic function on $(\Omega, g)$ such that $ u = c_i$ on $ S_i$, $ i = 1, 2$, 
where $ c_1$, $ c_2$ are constants with $ c_1 < c_2 < 1$.
For regular values $t$, let 
$$ A(t) = 8 \pi - \frac{1}{( 1 - t)} \int_{\Sigma_t}  H | \nabla u|  , \ \ \ \  \mathcal{A}(t) = \frac{A (t) }{1-t}   ,  $$
$$ B(t) =  4 \pi - \frac{1}{( 1 - t)^2} \int_{\Sigma_t}  | \nabla u|^2 , \  \ \  \ \mathcal{B} (t) =  \frac{B(t) }{1-t}   .  $$
Then, for any $  t_1 < t_2  $,
\be \label{eq-difference-mathcal-B}
\mathcal{B} (t_2) - \mathcal{B} (t_1) 
=   \int_{t_1}^{t_2}   \frac{1}{(1-t)^2}  \left[ 3 B(t) - A (t) \right] ,
\ee
and
\be \label{eq-difference-mathcal-A}
\begin{split}
\mathcal{A}(t_2) - \mathcal{A}(t_1) \ge & \ \int_{t_1}^{t_2}   \frac{1}{(1-t)^2}  \left[  3 B(t) - A (t)   
+  2 \pi  \left( 2 - \chi(\Sigma_t)  \right)  + \psi (t) \right]   ,
\end{split}
\ee
where
\bee \label{eq-def-psi-t}
\psi (t) =  \int_{\Sigma_t}  \left[  \frac34 \left( H - \frac{ 2 | \nabla u | }{ 1 - u } \right)^2 
+ | \nabla u |^{-2} | \nabla_{_{\Sigma_t} } | \nabla u | |^2 
+ \frac12 | \mathring{ \Pi } |^2 +  \frac12 R  \right] . 
\eee
As a result, if 
\begin{itemize}
\item  $(M, g)$ is a complete, orientable, asymptotically flat $3$-manifold 
with connected boundary $\Sigma$ and  $H_2 (M, \Sigma) = 0$; 
\item  $u$ is the harmonic function on $(M, g)$  with $ u = 0 $ at $ \Sigma$ and $ u \to 1$ at $\infty$; and
\item $g$ has nonnegative scalar curvature, 
\end{itemize}
 then $ \Sigma_t $ is connected, $ 3 B(t) - A (t) \ge 0 $  by \eqref{eq-A-and-3B}, and 
consequently, 
$$ \mathcal{B} (t_2) - \mathcal{B} (t_1)  \ge 0  \ \ \ \text{and} \ \ \
\mathcal{A} (t_2) - \mathcal{A} (t_1)  \ge 0 , \ \ \forall \ t_2 > t_1 .  $$
\end{prop}

\begin{proof}
Applying \eqref{eq-reg-for-B} in Lemma \ref{lem-reg} to $\Omega_{[t_1, t_2]} = \{ x \ | \ t_1 \le u(x) \le t_2 \}$,
one has
\bee
\begin{split}
   \int_{\Sigma_{t_2} } \frac{ | \nabla u |^2 }{(1- u )^3}   -  \int_{\Sigma_{t_1} } \frac{ | \nabla u |^2 }{(1- u )^3} 
= & \  \int_{\Omega_{ [t_1, t_2] } } \frac{ 3  | \nabla u |^3  }{ ( 1 - u )^4 }    + 
\int_{ \{ \nabla u \ne 0 \}  \subset \Omega_{[t_1, t_2]} } \frac{ \nabla^2 u ( \nabla u , \nabla u  ) } { (1 - u)^3  | \nabla u |  }   \\
= & \ \int_{t_1}^{t_2} \int_{\Sigma_t} \left[ \frac{3 | \nabla u |^2} {(1-t)^4 } - \frac{H | \nabla u |}{ ( 1 - t)^3 } \right]  ,
\end{split}
\eee
This, combined with $ \frac{1}{1-t_2} - \frac{1}{1 - t_1} = \int_{t_1}^{t_2} \frac{1}{( 1 - t)^2 } $,  shows
\be
\begin{split}
 \mathcal{B}(t_2) - \mathcal{B} (t_1)  
= & \ \int_{t_1}^{t_2} \left[   \frac{4\pi}{(1-t)^2} + \int_{\Sigma_t} \frac{H | \nabla u |}{ ( 1 - t)^3 }   -  \int_{\Sigma_t} \frac{3 | \nabla u |^2} {(1-t)^4 }  \right] \\
= & \ \int_{t_1}^{t_2}   \frac{1}{(1-t)^2}  \left[  3 B(t) - A (t)   \right] ,
\end{split}
\ee
which proves \eqref{eq-difference-mathcal-B}.

Similarly,  applying Lemma \ref{lem-reg-3} to $u$ on $\Omega_{[t_1, t_2]}$
and using \eqref{eq-hessian-square}, one has
\bee
\begin{split}
& \ \frac{1}{(1- t_2)^2} \int_{\Sigma_{t_2} } H  |\nabla u | 
- \frac{1}{(1- t_1)^2} \int_{\Sigma_{t_1} } H  |\nabla u |  \\
\le & \ \int_{t_1}^{t_2}   \frac{1}{(1-t)^2}  \left[ 2 \pi \chi(\Sigma_t)   +  \int_{\Sigma_t}   \frac{2  H | \nabla u | }{1-t}  
 \right. \\
& \  \hspace{2.5cm}  \left. - \int_{\Sigma_t} \frac34  H^2  
-  \int_{\Sigma_t}  \left(  | \nabla u |^{-2} | \nabla_{_{\Sigma_t} } | \nabla u | |^2 
+ \frac12 | \mathring{ \Pi } |^2 +  \frac12 R  \right)  \right]  \\
=  & \ \int_{t_1}^{t_2}   \frac{1}{(1-t)^2}  \left[  2 \pi \chi(\Sigma_t)  - \psi (t) 
 - \frac{1}{1-t} \int_{\Sigma_t} H | \nabla u |  +  \frac{3}{(1-t)^2 }  \int_{\Sigma_t} | \nabla u |^2   \right]  .
\end{split}
\eee
Therefore, 
\bee
\mathcal{A} (t_2) - \mathcal{A}(t_1) 
\ge   \int_{t_1}^{t_2}   \frac{1}{(1-t)^2}  \left[  3 B(t) - A (t)   +  2 \pi  \left( 2  - \chi(\Sigma_t) \right) + \psi(t) \right]   ,
\eee
which proves \eqref{eq-difference-mathcal-A}.
\end{proof}

\begin{rem}
As a corollary of \eqref{eq-difference-mathcal-B} and \eqref{eq-difference-mathcal-A}, 
\be \label{eq-monotone-AMO}
\begin{split}
& \ \left[ \mathcal{A}(t_2) - \mathcal{B}(t_2) \right] - \left[ \mathcal{A}(t_1) - \mathcal{B}(t_1) \right]  \\
\ge & \ \int_{t_1}^{t_2}   \frac{1}{(1-t)^2}  \left[  2 \pi  \left( 2 -  \chi (\Sigma_t) \right)  + \psi (t)   \right]  .
\end{split}
\ee
This corresponds to the monotonicity of $ F(t) = \mathcal{A}(t) - \mathcal {B}(t)$ in \cite{AMO21}.
\end{rem}

\end{document}